\documentclass[12pt]{article}

\usepackage{amsthm,amsmath,amsfonts,amssymb}
\usepackage[authoryear]{natbib}
\usepackage[colorlinks,citecolor=blue,urlcolor=blue]{hyperref}
\usepackage{graphicx}
\usepackage[a4paper, total={7.2in, 9.5in}]{geometry}

\usepackage[OT1]{fontenc}
\usepackage{dsfont}
\usepackage{stmaryrd}
\usepackage{microtype}

\theoremstyle{plain}

\newtheorem{theorem}{Theorem}[section]
\newtheorem{lemma}[theorem]{Lemma}

\theoremstyle{remark}
\newtheorem{definition}[theorem]{Definition}

\newtheorem{remark}[theorem]{Remark}

\newcommand{\sa}[1]{\ensuremath{\,{\buildrel #1 \over \sim}\,}}
\newcommand{\eq}[1]{\ensuremath{\,{\buildrel #1 \over =}\,}}
\newcommand{\err}{\mathrm{ME}}
\newcommand{\adj}{\mathrm{Adj}}
\newcommand{\lap}{\mathrm{Lap}}
\newcommand{\cusum}{\mathrm{CU}}

\newcommand{\cP}{\mathcal{P}}
\renewcommand{\P}{\mathbb{P}}
\newcommand{\E}{\mathbb{E}}
\newcommand{\bern}{\mathrm{Bernoulli}}
\newcommand{\var}{\mathbb Var}
\newcommand{\Bkt}[1]{\left\llbracket #1\right\rrbracket}
\newcommand{\bkt}[1]{\llbracket #1\rrbracket}
\newcommand{\pre}{\preccurlyeq}
\newcommand{\cR}{\mathcal{R}}

\title{\textbf{Concentration inequalities for correlated network-valued processes with applications to community estimation and changepoint analysis}}

\usepackage{authblk}

\author[1]{Sayak Chatterjee\thanks{Email: \texttt{sayakchatterjee999@gmail.com}}}
\author[2]{Shirshendu Chatterjee\thanks{Email: \texttt{shirshendu@ccny.cuny.edu}. Research supported by PSC-CUNY Enhanced Research Award \#64673-00 52.}}
\author[3, 4]{Soumendu Sundar Mukherjee\thanks{Email: \texttt{soumendu041@gmail.com}. Research supported by an INSPIRE Faculty Fellowship from the Department of Science and Technology, Government of India.}}
\author[1]{Anirban Nath\thanks{Email: \texttt{anirbannath98@gmail.com}}}
\author[5]{Sharmodeep Bhattacharyya\thanks{Email: \texttt{bhattash@science.oregonstate.edu}}}

\affil[1]{Indian Statistical Institute, Kolkata}
\affil[2]{Department of Mathematics, City College and Graduate Center, City University of New York}
\affil[3]{Interdisciplinary Statistical Research Unit, Indian Statistical Institute, Kolkata}
\affil[4]{Department of Mathematics, National University of Singapore}
\affil[5]{Department of Statistics, Oregon State University}

\date{}

\usepackage{fancyhdr,ifthen}
\begin{document}

\maketitle

\begin{abstract}
Network-valued time series are currently a common form of network data. However, the study of the aggregate behavior of network sequences generated from network-valued stochastic processes is relatively rare. Most of the existing research focuses on the simple setup where the networks are independent (or conditionally independent) across time, and all edges are updated synchronously at each time step. In this paper, we study the concentration properties of the aggregated adjacency matrix and the corresponding Laplacian matrix associated with network sequences generated from lazy network-valued stochastic processes, where edges update asynchronously, and each edge follows a lazy stochastic process for its updates independent of the other edges. We demonstrate the usefulness of these concentration results in proving consistency of standard estimators in community estimation and changepoint estimation problems. We also conduct a simulation study to demonstrate the effect of the laziness parameter, which controls the extent of temporal correlation, on the accuracy of community and changepoint estimation.
\end{abstract}

\section{Introduction}
\label{sec:intro}
Sequences of networks are now widely available as the main observable or as derived data in many fields of research, including time series of social networks \citep{panisson2013fingerprinting, stopczynski2014measuring, rocha2010information}, epidemiological networks \citep{salathe2010high, rocha2011simulated}, animal networks \citep{gates2015controlling, lahiri2007structure},  mobile and online communication networks \citep{krings2012effects, ferraz2015rsc, jacobs2015assembling}, economic networks \citep{popovic2014extraction, zhang2014dynamic}, brain networks \citep{park2013structural, sporns2013structure}, genetic networks \citep{rigbolt2011system}, and ecological networks \citep{blonder2012temporal}, to name a few. One of the main forms of network sequence data is temporal networks or network-valued time series. Nevertheless, most existing models of network-valued time series assume independent or conditionally independent network layers in the time series. Such models cannot adequately capture the complexity of possible inter-layer dependencies observed in real-world network time series data.

In this paper, we consider network time series data generated from network-valued stochastic processes, where the model allows for direct dependence between present and future network-layers. We provide rigorous concentration results for the aggregated adjacency and Laplacian matrices of such time series data around their population counterparts. Such concentration results have applications in numerous statistical questions, e.g., estimation of the parameters of the underlying process, estimation of latent states in case of an underlying latent space structure, changepoint estimation in case of non-stationarity or structural changes, community detection/estimation in case of the presence of a (stationary) community structure, etc. In this paper, we only focus on the problems of changepoint and community estimation. Exploration of the rest of the applications is left for future work.

We now discuss the main contribution of this paper. Consider a sequence of $T (\ge 2)$ adjacency matrices $(A^{(1)}, \ldots, A^{(T)})$. We posit that the sequence is generated from a lazy stochastic process whose snapshots are inhomogeneous Erd\H{o}s-R\'{e}nyi random graphs.

\begin{definition}[Lazy inhomogeneous Erd\H{o}s-R\'{e}nyi process]\label{def:lirgm}
A sequence of $T \ge 2$ symmetric adjacency matrices $(A^{(1)}, \ldots, A^{(T)})$ on the vertex set $[n] := \{1, \ldots, n\}$ is said to be generated from the \emph{lazy inhomogeneous Erd\H{o}s-R\'{e}nyi process} (abbreviated \emph{lazy IER process}) with mean adjacency matrix $P=((p_{ij})) \in [0, 1]^{n \times n}$ and laziness (or stickiness) parameter $\alpha \in (0, 1)$, if $A_{ij}^{(1)} \sim \bern(p_{ij})$, and, given $A^{(t - 1)}$,
\begin{equation}\label{eq:Agen}
 A_{ij}^{(t)} \,\,
\begin{cases}
    =A_{ij}^{(t-1)} & \text{ with probability } \alpha, \\
    \sim \bern(p_{ij}) & \text{ with probability } 1-\alpha,
\end{cases}    
\end{equation}
for all $t\in\{2, 3, \ldots, T\}$ and $1\le i\le j\le n$. Moreover, the edges $A_{i,j}^{(t - 1)}$ are conditionally independent of each other given $A^{(t - 1)}$.
\end{definition}

We define the aggregated adjacency and Laplacian matrices of the network sequence $(A^{(1)}, \ldots, A^{(t)})$ as follows.
\begin{equation}
\label{eq:aggAL}
A :=\sum_{t\in[T]} A^{(t)} \quad\text{ and } \quad \mathcal L:= I_n-D^{-1/2}AD^{-1/2}.
\end{equation}
Here $A$ denotes the aggregated adjacency matrix and $\mathcal{L}$ is the associated normalized {\it graph Laplacian}, with $D$ being the diagonal matrix of aggregated degrees $d_i :=\sum_{t\in[T], j\in[n]}A^{(t)}_{ij}$, $i = 1, \ldots, n$. The population versions of the aggregated adjacency and Laplacian matrices are defined respectively as
\begin{equation}
\label{eq:EAggAL}
\bar A :=\mathbb E(A) \quad\text{ and } \quad \bar{\mathcal L}:= I_n-\bar D^{-1/2}\bar A\bar D^{-1/2},
\end{equation}
where $\bar D = \E D$ is the diagonal matrix with $\bar d_i :=T\sum_{j\in[n]}p_{ij}$ as its $i$-th diagonal entry. We are now ready to give a sketch of the main theorem of this paper.
\begin{theorem}
\label{thm:main_sketch}
Let $(A^{(t)}, t\in[T])$ be the adjacency matrices of a sequence of $T$ networks on the vertex set $[n]$ that follows the lazy IER process with parameters $P_{n\times n}$ and $\alpha$. Let $d_{\min}:=\min_{i\in[n]}\sum_{j\in[n]}p_{ij}$ (resp.~$d_{\max}:=\max_{i\in[n]}\sum_{j\in[n]}p_{ij}$) denote the minimum (resp.~maximum) among the expected degrees of the vertices for each $t\in[T].$ Then there exist constants $C, C_1(\alpha) > 0$ such that if $Td_{\max}>C(\log(n))^3$, then
\[
    \|A-\bar A\| \le C_1(\alpha)\sqrt{Td_{\max}\log(n)}
\]
with high probability. Moreover, there exist constants $C, C_2(\alpha) > 0$ such that if $Td_{\min}>C(\log(n))^3$, then
\[
    \|\mathcal L-\bar{\mathcal L}\| \le C_2(\alpha)\sqrt{\frac{\log(n)}{Td_{\min}}}
\]
with high probability.
\end{theorem}
Here $\|\cdot\|$ denotes the matrix operator norm. A more precise version of Theorem \ref{thm:main_sketch} is given in Theorem \ref{thm:main} of \textsection \ref{sec:setup}.

\begin{remark}
Both the constants $C_1(\alpha)$ and $C_2(\alpha)$ that appear in Theorem \ref{thm:main_sketch} are increasing functions of $\alpha$. Moreover, $C_1(\alpha), C_2(\alpha)\uparrow \infty$ as $\alpha\uparrow 1$, and $C_1(\alpha), C_2(\alpha)\downarrow C_0$ for an universal constant $C_0 > 0$ as $\alpha\downarrow 0$. In fact, both of these constants are $\Omega(1/\sqrt{1 - \alpha})$.
\end{remark}

\begin{remark}
When $\alpha = 0$, all edges at each layer are generated independently of previous layers. Then the aggregated adjacency and Laplacian matrices concentrates the most compared to other values of $\alpha$. As $\alpha$ increases to 1, the correlation among the edges between each pair of vertices across all layers increases to 1. Consequently, the concentration properties of the aggregated adjacency and Laplacian matrices deteriorate.
\end{remark}

\begin{remark}
The main difficulty in proving the desired concentration inequality for the aggregated adjacency and Laplacian matrices is the complexity of the correlation structure among different layers. Standard methods and known approaches for proving matrix concentration (e.g., matrix Bernstein-type inequalities \citep{oliveira2009concentration, tropp2015introduction}, combinatorial arguments \citep{lei2015consistency, bhattacharyya2020consistent}, path counting arguments used in random matrix theory \citep{Lu_2013}) are not useful for proving nontrivial concentration results for strongly correlated layers of multi-layer network models that we address in Theorem \ref{thm:main_sketch}. Our approach enables us to go beyond weakly correlated multi-layer network models (including multi-layer network models with conditionally independent layers), and develop tools to analyze multi-layer network models involving more complex correlation structures so that computationally efficient and provably consistent algorithms can be developed to address network inference problems for such complex multi-layer networks. 
\end{remark}

To the best of our knowledge, Theorem \ref{thm:main_sketch} is the first such result on aggregated dependent network layers from a network-valued time series.

\subsection{Related works}
\label{sec:related_works}
Most of the probabilistic models for multiple networks appearing in the literature are extensions of random graph models for single networks to the multiple networks setup. Examples of such models include extensions of latent space models \citep{sarkar2005dynamic, sewell2015latent}, mixed membership block models \citep{ho2011evolving, padilla2019change}, random dot-product models \citep{tang2013attribute}, stochastic block models \citep{xu2014dynamic, xu2015stochastic, matias2017statistical, ghasemian2016detectability, corneli2016exact, zhang2017random, pensky2019dynamic, bhattacharyya2020consistent}, and Erd\'{o}s-R\'{e}nyi models \citep{crane2015time}. Some Bayesian models and associated inference procedures have also been proposed in the context of multiple networks \citep{yang2011detecting, durante2014nonparametric}. Most of these models consider network sequences that are either independent (e.g., \cite{xu2015stochastic, bhattacharyya2020consistent}) or conditionally independent (e.g., \cite{matias2017statistical, padilla2019change, pantazis2022importance, athreya2022discovering}). Of these, the work that is the closest to ours is \cite{padilla2019change}. However, the dependence structure in \cite{padilla2019change} is in the form of a lazy process in the latent space of random dot product graphs, so that the network layers, though correlated, are conditionally independent given the latent variables. In contrast, the dependence structure between the network layers in our model stems from lazy Markov chains on individual edges as given in Definition \ref{def:lirgm}. As this structure does not admit a conditionally independent temporal representation, the theoretical study of the proposed model is much more challenging. Indeed, the proofs of the concentration inequalities given in Theorem~\ref{thm:main_sketch} rely on rather delicate decoupling arguments.

\subsection{Our contributions}
\label{sec:contributions}
We propose a novel lazy network-valued stochastic process in this paper as a generating model for temporal network data. The main theoretical contributions of our work are as follows.
\begin{enumerate}
    \item[(a)] The principal contributions of this paper are the concentration results for the aggregated adjacency and Laplacian matrices of network-valued time series generated from the lazy IER process as given in Theorem \ref{thm:main_sketch}. The rates of concentration depend on the amount of dependence across network layers which is captured by the laziness parameter ($\alpha$) of the lazy IER process. 
    
    To the best of our knowledge, the proofs of these concentration results are completely new. They involve tackling sums of dependent random variables and require delicate decoupling arguments. The precise statements of these concentration results are given in \textsection \ref{sec:setup}.
    
    \item[(b)] As an application of the concentration results, we provide consistency results for standard spectral clustering methods to recover community labels from network-valued time series generated from lazy IER processes with community structures. Again, the novelty here is that the theoretical results on consistency are derived for network sequences with dependent network layers (as these are generated from  a lazy stochastic block model (SBM) process with direct dependence between network layers). As expected, the consistency rates depend on the amount of dependence between the network layers as captured by the laziness parameter $\alpha$ of the generating lazy SBM process. The specifics of the model used and corresponding theoretical results are presented in \textsection \ref{sec:community}.
    
    \item[(c)] As another application of the concentration results, we also show consistency of standard cumulative sum (CUSUM) based changepoint estimation procedures for network-valued time series generated from a piecewise lazy IER process. Again, the consistency rates are shown to depend on the amount of dependence between the network layers as captured by the laziness parameter $\alpha$. The details of the piecewise lazy IER process and the consistency results are presented in \textsection \ref{sec:changepoint}.
\end{enumerate}
Empirical results in support of our theoretical results are presented in \textsection \ref{sec:empirical-results}.

\section{Main results}
\label{sec:setup}
In this paper, we suppose that the observed data is an (undirected) \emph{multiplex network} that consists of a finite sequence of $T$ unlabeled graphs $\left(G^{(t)}; t\in[T]\right)$ on the same vertex set $[n] := \{1,  2, \ldots, n\}$. $G^{(t)}$ is referred to as the $t$-th \emph{network layer}, which is represented by the corresponding adjacency matrix $A^{(t)}\in\{0, 1\}^{n\times n}$, where $A^{(t)}_{ij} = 1$ (resp.~$0$) if $i$ and $j$ are (resp.~not) connected by an edge in the $t$-th layer. Thus the observed data consists of $T \ge 2$ adjacency matrices $(A^{(1)}, \ldots, A^{(T)})$. 

As in Definition \ref{def:lirgm}, we say that the sequence of $T\ge 2$ networks $(G^{(1)}, \ldots,  G^{(T)})$ on the vertex set $[n]$ follows the lazy IER process with parameters $P=((p_{ij}))\in[0,1]^{n\times n}$ satisfying $p_{ij}=p_{ji}$ for all $i, j\in[n]$ and $\alpha \in (0,1)$, if the associated adjacency matrices satisfy the generating mechanism of Equation \ref{eq:Agen}. Here, $\alpha$ represents the laziness (or stickiness) parameter, and $P$  represents the mean-adjacency matrix. Let the aggregated adjacency matrix, $A$, and the associated {\it graph Laplacian} matrix, $\mathcal L$, be as defined in Equation \ref{eq:aggAL} and the corresponding population versions, $\bar A$ and $\bar{\mathcal{L}}$ respectively, be as defined in Equation \ref{eq:EAggAL}. Then the main theorem of this paper can be rigorously stated as follows.
\begin{theorem} \label{thm:main}
Let $(A^{(t)}, t\in[T])$ be the adjacency matrices of a sequence of $T$ networks on the vertex set $[n]$ that follows the lazy IER process with parameters $P$ and $\alpha$.  For any constant $c >0$ there exists another constant $C = C(c)>0$, which is independent of $n, T, P$ and $\alpha$, such that the following holds. Let $d_{\min}:=\min_{i\in[n]}\sum_{j\in[n]}p_{ij}$ (resp.~$d_{\max}:=\max_{i\in[n]}\sum_{j\in[n]}p_{ij}$) denote the minimum (resp.~maximum) among the expected degrees of the vertices for each $t\in[T]$.
\begin{enumerate}
    \item[(A)] If $Td_{\max}>C(\log(n))^3$, then there is a constant $C_1(\alpha,c)$ such that for any $\delta\in(n^{-c}, 1/2)$,
    \[
        \mathbb P\bigg(\|A-\bar A\|\le C_1\sqrt{Td_{\max}\log(n/\delta)}\bigg)\ge 1-\delta.
    \]
    \item[(B)] Moreover, if $Td_{\min}>C(\log(n))^3$, then there is a constant $C_2(\alpha)$ such that  for any $\delta\in(n^{-c}, 1/2)$,
    \[
        \mathbb P\bigg(\|\mathcal L-\bar{\mathcal L}\|\le C_2\sqrt{\frac{\log(4n/\delta)}{Td_{\min}}}\bigg)\ge 1-\delta.
    \]
\end{enumerate}
\end{theorem}
The proof of Theorem \ref{thm:main} is given in \textsection \ref{sec:proof}.

\subsection{Two applications}
\label{sec:applications}
We demonstrate the usefulness of our concentration results by deriving consistency guarantees in two important statistical problems in network analysis: (i) community estimation, (ii) changepoint estimation. Both of these problems have been studied in great detail in the setup of independent networks. Although some recent papers study changepoint problems in the context of dependent network-valued time series, to the best of our knowledge there has not been any study of community estimation in the context of dependent dynamic networks. 

\subsubsection{Community estimation}
\label{sec:community}
Suppose there is an underlying community structure in our network-valued time series, where each vertex belongs to one of $K$ communities. The community structure is captured by a community membership matrix $Z \in \{0, 1\}^{n \times K}$, where $Z_{ia} = 1$ if and only if vertex $i$ belongs to community $a \in [K]$. Consider a sequence of networks $G^{(1)}, \ldots, G^{(T)}$ generated from a lazy IER process, with common mean matrix $P = ZBZ^\top$, where $B \in [0, 1]^{K \times K}$ is a matrix whose entries specify the between/within community edge-formation probabilities, i.e. the probability of a link between a vertex from community $a$ and a vertex from community $b$ is given by $B_{ab}$. We call a lazy IER process with such community structure a \emph{lazy SBM process}. Our goal is to recover the community membership matrix $Z$. To that end, we use the spectral clustering algorithm on (a) the aggregated adjacency matrix $A$, and (b) aggregated Laplacian matrix $\mathcal{L}$. Concentration of these matrices around their expectation enables us to recover $Z$ via an application of the Davis-Kahan theorem. Spectral clustering uses $k$-means in its last step. We use the $(1 + \epsilon)$-approximate $k$-means algorithm as in \cite{lei2015consistency}.

The misclustering error of a community estimator $\hat{Z}$ is defined to be
\[
    \err(\hat{Z}, Z) := \inf_{\Pi \in \mathcal{P}_K} \frac{1}{n}\|\hat{Z} - Z\Pi \|_F^2,
\]
where $\cP$ is the set of all $K \times K$ permutation matrices.

We denote by $\hat{Z}_{\adj}$ and $\hat{Z}_{\lap}$ the community estimates obtained by applying $(1 + \epsilon)$-approximate spectral clustering to $A$ and $\mathcal L$, respectively.

\begin{theorem}\label{thm:comm}
Let $d_{\max}$ and $d_{\min}$ be as in Theorem~\ref{thm:main}. Let $\gamma_{\adj}$ denote the smallest non-zero singular value of $P = ZBZ^\top$. For any constant $c >0$ there exists another constant $C = C(c)>0$ such that if $Td_{\max} > C(\log (n))^3$, then there is a constant $C_1 = C_1(\alpha, c)$ such that for any $\delta\in(n^{-c}, 1/2)$,
\[
    \err(\hat{Z}_{\adj}, Z) \le \frac{C_1 (2 + \epsilon) K d_{\max} \log (n / \delta)}{\gamma_{\adj}^2 T}
\]
with probability at least $1 - \delta$.

Also, suppose $\gamma_{\lap}$ denotes the smallest non-zero singular value of $D_P^{-1/2} P D_P^{-1/2}$, where $D_P$ is a diagonal matrix with $i$-th diagonal entry $\sum_{j \in [n]} p_{ij}$. For any constant $c >0$ there exists another constant $C = C(c)>0$ such that if $T d_{\min} > C(\log (n))^3$, then there is a constant $C_2 = C_2(\alpha, c)$ such that for any $\delta\in(n^{-c}, 1/2)$,
\[
    \err(\hat{Z}_{\lap}, Z) \le \frac{C_2 (2 + \epsilon)  K \log(4n/\delta)}{\gamma_{\lap}^2 T d_{\min}}
\]
with probability at least $1 - \delta$.
\end{theorem}
The constants $C_1$ and $C_2$ in Theorem~\ref{thm:comm} are the same (up to absolute multiplicative constant factors) as their namesakes appearing in Theorem~\ref{thm:main}.

We note here that in case of independent networks, estimators based on spectral clustering on the aggregated adjacency matrix were analysed in \cite{bhattacharyya2020general}. To the best of our knowledge, spectral clustering on the aggregated Laplacian matrix has not been analysed before, even in the context of independent networks. Theorem~\ref{thm:comm} provides such an analysis in a more general dependent set-up.

\subsubsection{Changepoint estimation}
\label{sec:changepoint}
To keep things simple, we consider the setup of a single changepoint, where we observe a sequence of graphs $G^{(1)}, \ldots, G^{(T)}$ such that for some $1 \le \tau \le T - 1$,
\begin{enumerate}
    \item[(a)] $G^{(1)}, \ldots, G^{(\tau)}$ are from a lazy IER process with mean matrix $P$.
    \item[(b)] $G^{(\tau + 1)}, \ldots, G^{(T)}$ are from a lazy IER process with mean matrix $Q$.
    \item[(c)] $A_{ij}^{(\tau + 1)}$ equals $A_{ij}^{(\tau)}$ with probability $\alpha$, and is sampled independently (of $A^{(1:\tau)}$ and across $i, j$) from Bernoulli($Q_{ij}$) with probability $(1 - \alpha)$. 
\end{enumerate}

Notice that under this model,
\[
    \E A^{(t)} = \begin{cases}
        P & \text{if } t \le \tau \\
        Q & \text{if } t > \tau.
    \end{cases}
\]
In other words, there is a change in the link probability matrix at epoch $\tau$.

To estimate the changepoint $\tau$, we use the CUSUM statistic on the sequence of adjacency matrices. Let $\xi \in [0, 1]$. Define
\[
    S_{\xi}(t) := \bigg(\frac{t}{T}\bigg(1 - \frac{t}{T}\bigg)\bigg)^{\xi} \bigg\|\frac{1}{t}\sum_{s = 1}^t A^{(s)} - \frac{1}{T - t}\sum_{s = t + 1}^T A_s \bigg\|,
\]
where recall that $\|\cdot\|$ denotes the matrix operator norm. The CUSUM estimate of $\tau$ then is given by
\begin{equation}\label{eq:cusum_stat}
    \hat{\tau}_{\cusum}^{(\Lambda, \xi)} \in \arg \max_{\Lambda \le t \le T - \Lambda} S_{\delta}(t),
\end{equation}
where $\Lambda$ is known a priori. This statistic \eqref{eq:cusum_stat} has been analysed earlier in the setup of independent networks in \cite{mukherjee2018thesis} and \cite{bhattacharyya2020consistent}.

\begin{theorem}\label{thm:cpe}
Let $\kappa = \min\{\tau, T - \tau\}$. For any $c > 0$, there exists a constant $C(c) > 0$ and a constant $C_1 =C_1(\alpha, c) > 0$ such that for any $\Lambda \le \kappa$,  if $ \Lambda \cdot (\min \{ d_{\max}^P, d_{\max}^Q \}) > C (\log n)^3$, then, for any $0 < c' < c$, 
\[
    |\hat{\tau}_{\cusum}^{(\Lambda, \xi)} - \tau| \le \frac{4 C_1 T}{\big(\frac{\tau(T - \tau)}{T^2}\big)^{\xi}\|P - Q\|} \sqrt{\frac{\max(d^P_{\max}, d^Q_{\max})}{\Lambda} (1 + c')\log n}
\]
with probability at least $1 - 6Tn^{-c'}$.
\end{theorem}
The constant $C_1$ in Theorem~\ref{thm:cpe} is the same (up to an absolute multiplicative constant factor) as its namesake appearing in Theorem~\ref{thm:main}.

\section{Empirical results} \label{sec:empirical-results}

\subsection{Studying the concentration} \label{sec:concentration_simulation}
We consider two sequences of $T = 30$ networks with $n = 500$ vertices generated from two different lazy IER processes. One of them has mean matrix from a stochastic block model (SBM) and another has a graphon as its mean matrix. For the stochastic block model, there are two communities of equal sizes with intra-cluster link formation probability $a/n$ and inter-cluster link formation probability $b/n$. For the graphon model, we simulated $U_i\sim\mathrm{Unif}(0,1)$ i.i.d. for $i\in[n]$ and put $p_{ij} = W(U_i,\,U_j)$ for $1\le i < j \le n$ where $W(x, y) = \frac{1}{k(1 + \exp(-x - y))}.$ We vary the average degree $d_{\mathrm{avg}}$ of the SBM as $d_{\mathrm{avg}} = 4, 5, 6$, for which we take $(a, b) = (7, 1), (8, 2), (9, 3)$, respectively. For the graphon model, to keep the average degree close to the same values, we take $k = 60, 72, 90$, respectively. The quantities $\|A-\bar{A}\|$ and $\|\mathcal{L}-\bar{\mathcal{L}}\|$ in Theorem \ref{thm:main} are appropriately scaled to $\|A-\bar{A}\|/\sqrt{Td_{\mathrm{avg}}}$ and $\|\mathcal{L}-\bar{\mathcal{L}}\|\sqrt{Td_{\mathrm{avg}}}$, respectively, and are plotted against the stickiness parameter $\alpha$. Figures \ref{fig:conc_sbm} and \ref{fig:conc_graphon} show the plots for the SBM and the graphon model, respectively. The plots confirm that the constants $C_1$ and $C_2$ in Theorem \ref{thm:main} are $\Omega(1/\sqrt{1-\alpha})$ as predicted by our theory.

\begin{figure}[!t]
    \centering
    \begin{tabular}{cc}
    \includegraphics[scale = 0.6]{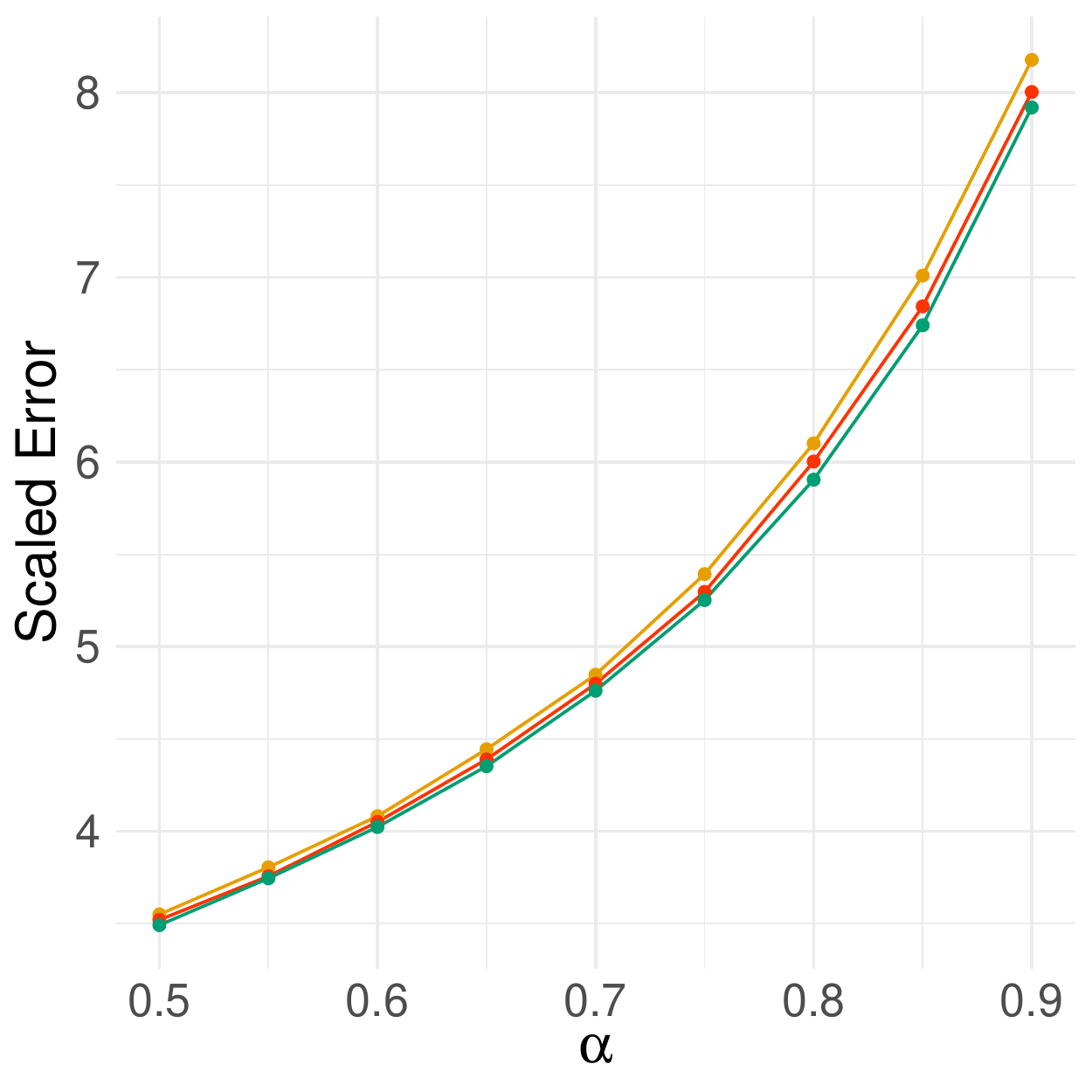} & \includegraphics[scale = 0.6]{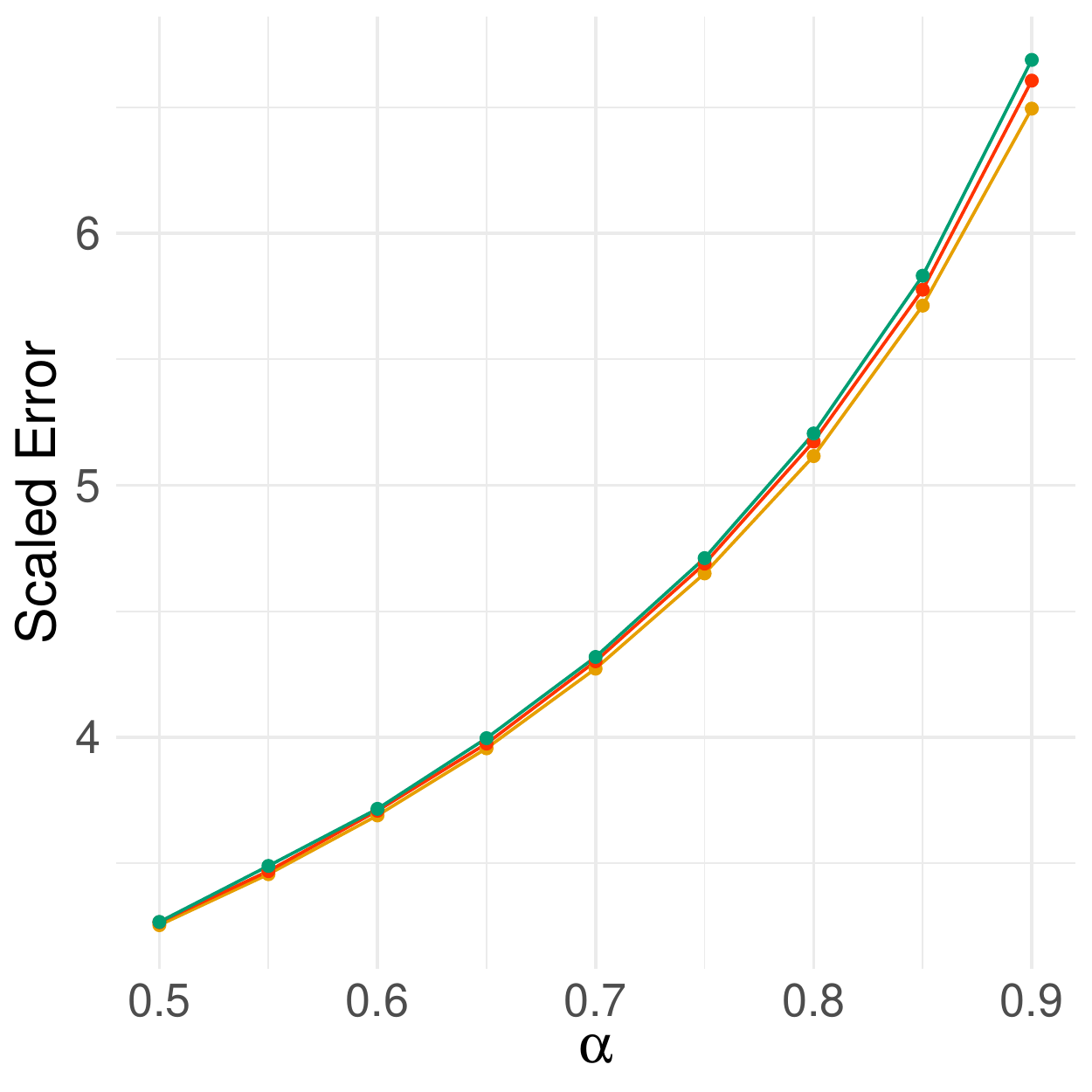}
    \end{tabular}
    \caption{Concentration plots for SBM. Left panel: aggregated adjacency. Right panel: aggregated Laplacian. Parameters: $a=7,\,b=1$ (yellow); $a=8,\,b=2$ (red); $a=9,\,b=3$ (green).}
    \label{fig:conc_sbm}
\end{figure}

\begin{figure}[!ht]
    \centering
    \begin{tabular}{cc}
    \includegraphics[scale = 0.5]{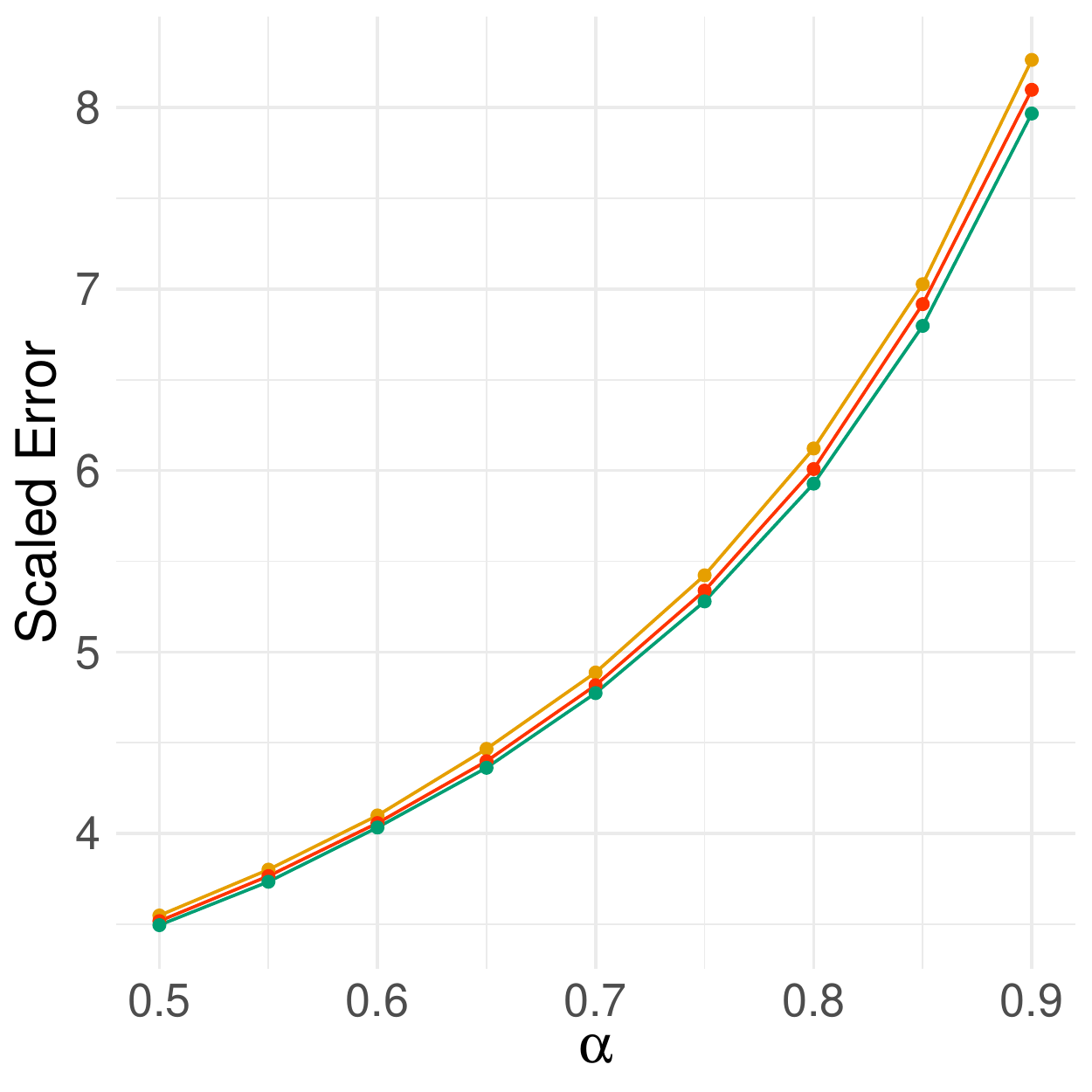} & \includegraphics[scale = 0.5]{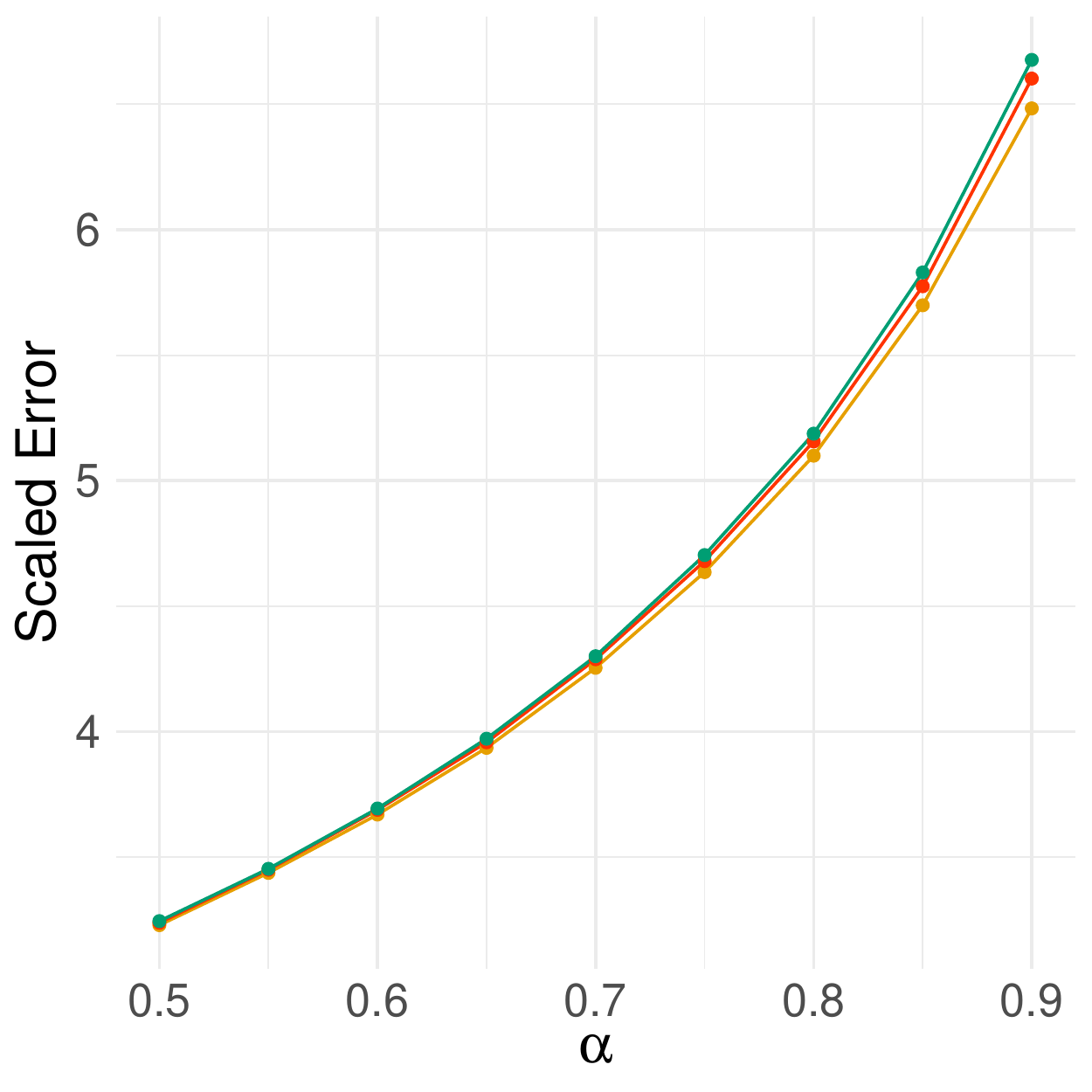}
    \end{tabular}
    \caption{Concentration plots for graphon model. Left panel: aggregated adjacency. Right panel: aggregated Laplacian. Parameters: $k=60$ (green), $72$ (red), 90 (yellow).}
    \label{fig:conc_graphon}
\end{figure}

\subsection{Community estimation} \label{sec:community_simulation}
We consider sequences of $T = 30$ networks generated from a lazy SBM process on $n = 500$ vertices with two communities, intra-cluster link formation probability $a/n$ and inter-cluster link formation probability $b/n$. The two communities have equal sizes. The signal to noise ratio in this model is given by $\Gamma = (a - b)^2/(2(a + b))$. In Figure \ref{fig:comm_est}, we plotted the misclustering error as a function of the stickiness parameter $\alpha$, in three different settings of $a$ and $b$ (and hence $\Gamma$). The error, as expected, increases with $\alpha$ due to lack of concentration. Also, the misclustering error is smaller for larger $\Gamma$. 

\begin{figure}[!ht]
    \centering
    \begin{tabular}{cc}
    \includegraphics[scale = 0.6]{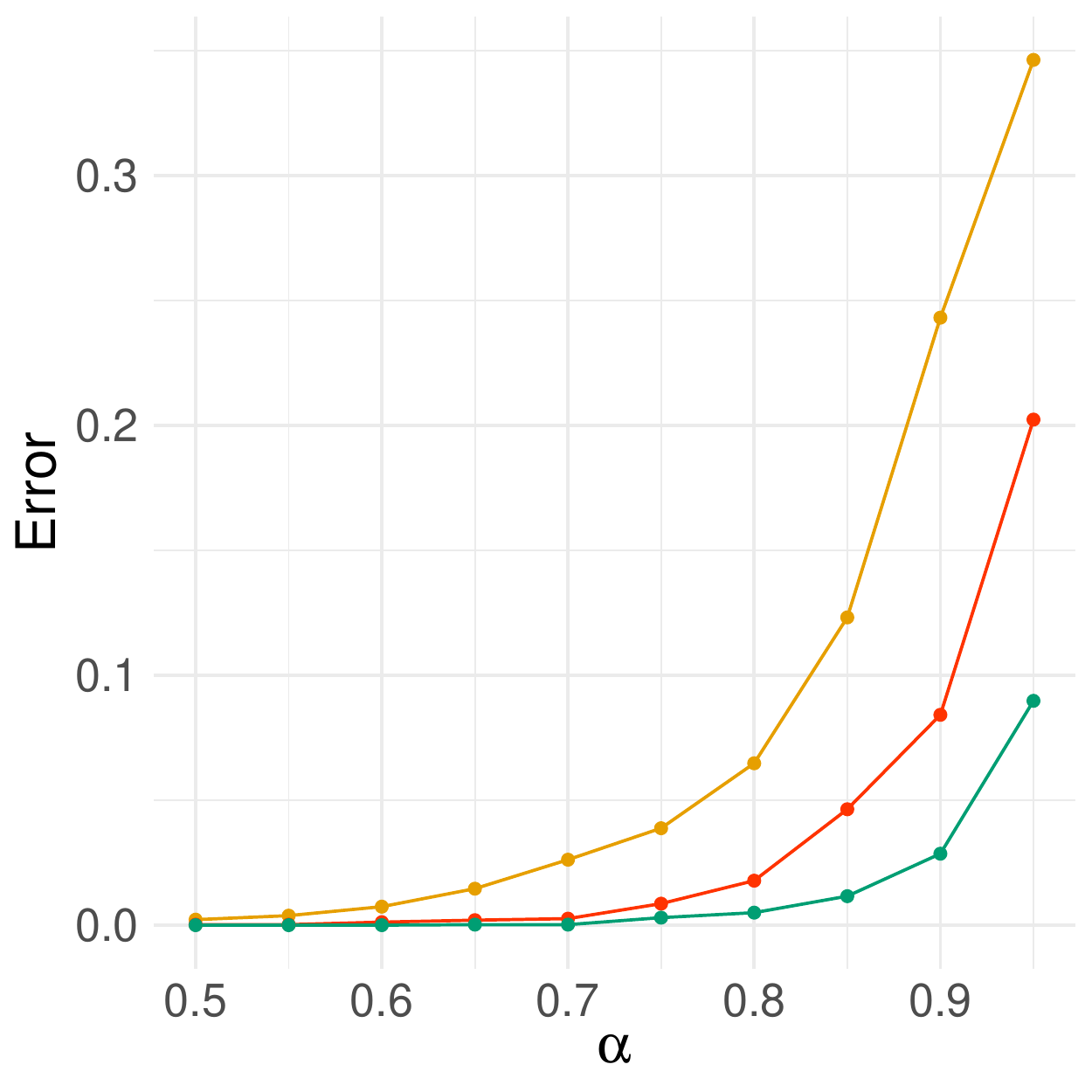} & \includegraphics[scale = 0.6]{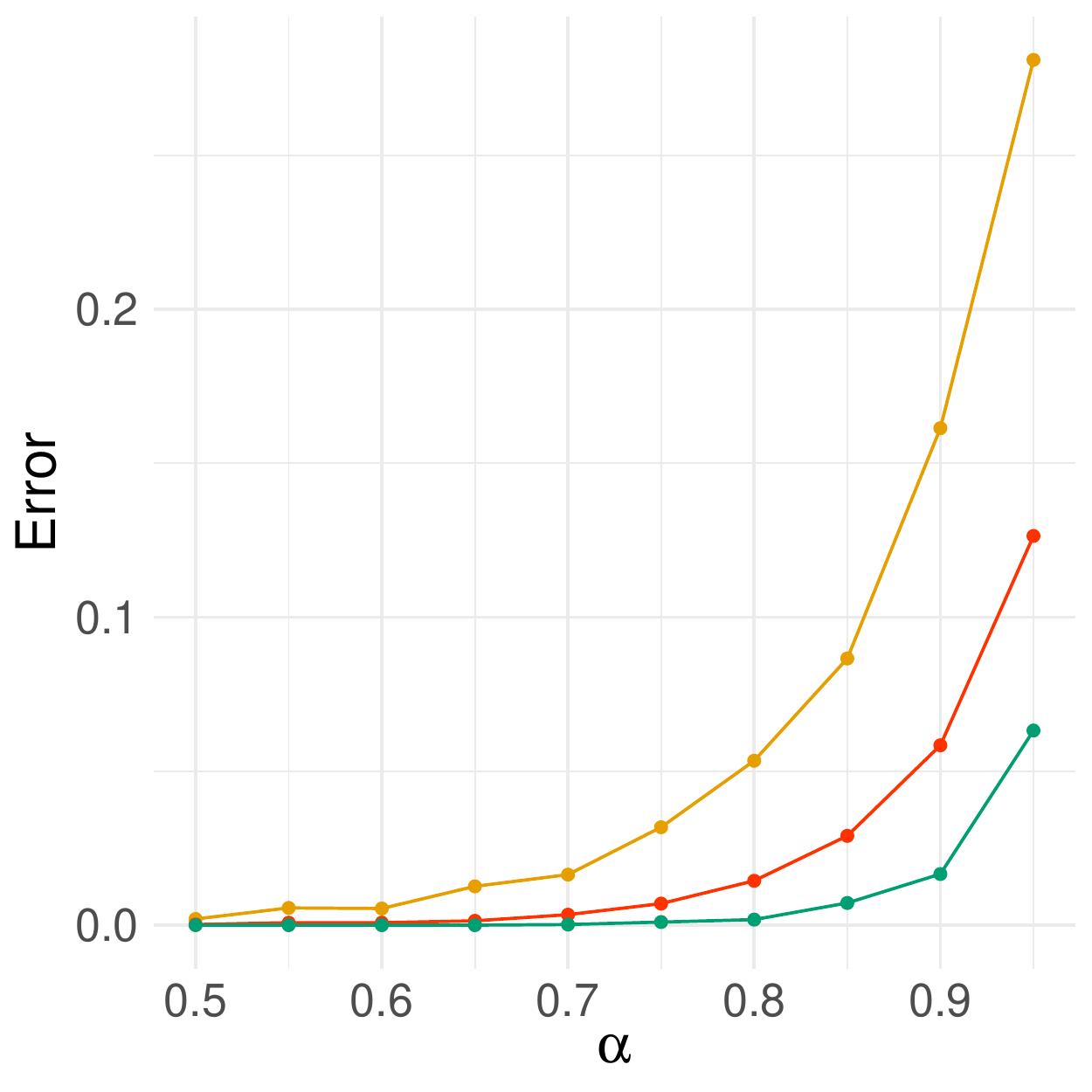}
    \end{tabular}
    \caption{Average misclustering error in community estimation for SBM over 100 experiments. Left panel: using the aggregated adjacency matrix. Right panel: using the aggregated Laplacian matrix. Parameters: $a=7,\,b=3$ (yellow); $a=7.5,\,b=2.5$ (red); $a=8,\,b=2$ (green).}
    \label{fig:comm_est}
\end{figure}

\subsection{Changepoint estimation} \label{sec:changepoint_simulation}
For illustrating changepoint estimation, we simulate networks from lazy IER process with mean matrices as two different graphons. We generate $T = 30$ networks with $n = 500$ vertices with the changepoint at $\tau = 15$. For the first $15$ networks, we start with $U_i\sim\mathrm{Unif}(0,1)$ i.i.d. for $i\in[n]$ and put $p_{ij} = W(U_i, U_j)$ for $1 \le i < j \le n$ where $W(x, y) = \frac{1}{72(1 + \exp(-x - y))}$ to keep the average degree close to $5$ and simulate the networks for a fixed stickiness parameter $\alpha$. As for the mean matrix of the post-change networks, we randomly select $\Delta_e$ pairs of vertices and increase the corresponding edge probabilities of the pre-change mean matrix by $\Delta_p$. We generate the rest of the process using this new mean matrix. We vary $\Delta_e \in \{100, 110, 120\}$ and $\Delta_p \in \{0.05, 0.07\}$. We use the CUSUM statistic with $\xi = 1/2$ for estimating the changepoint. In Figure \ref{fig:change_pt_est}, we plot the average absolute error, i.e. the average absolute difference between the true and the estimated changepoints, as a function of the stickiness parameter $\alpha$. As expected, the average absolute error increases with the stickiness parameter $\alpha$ due to lack of concentration.

\begin{figure}[!ht]
    \centering
    \begin{tabular}{cc}
    \includegraphics[scale = 0.6]{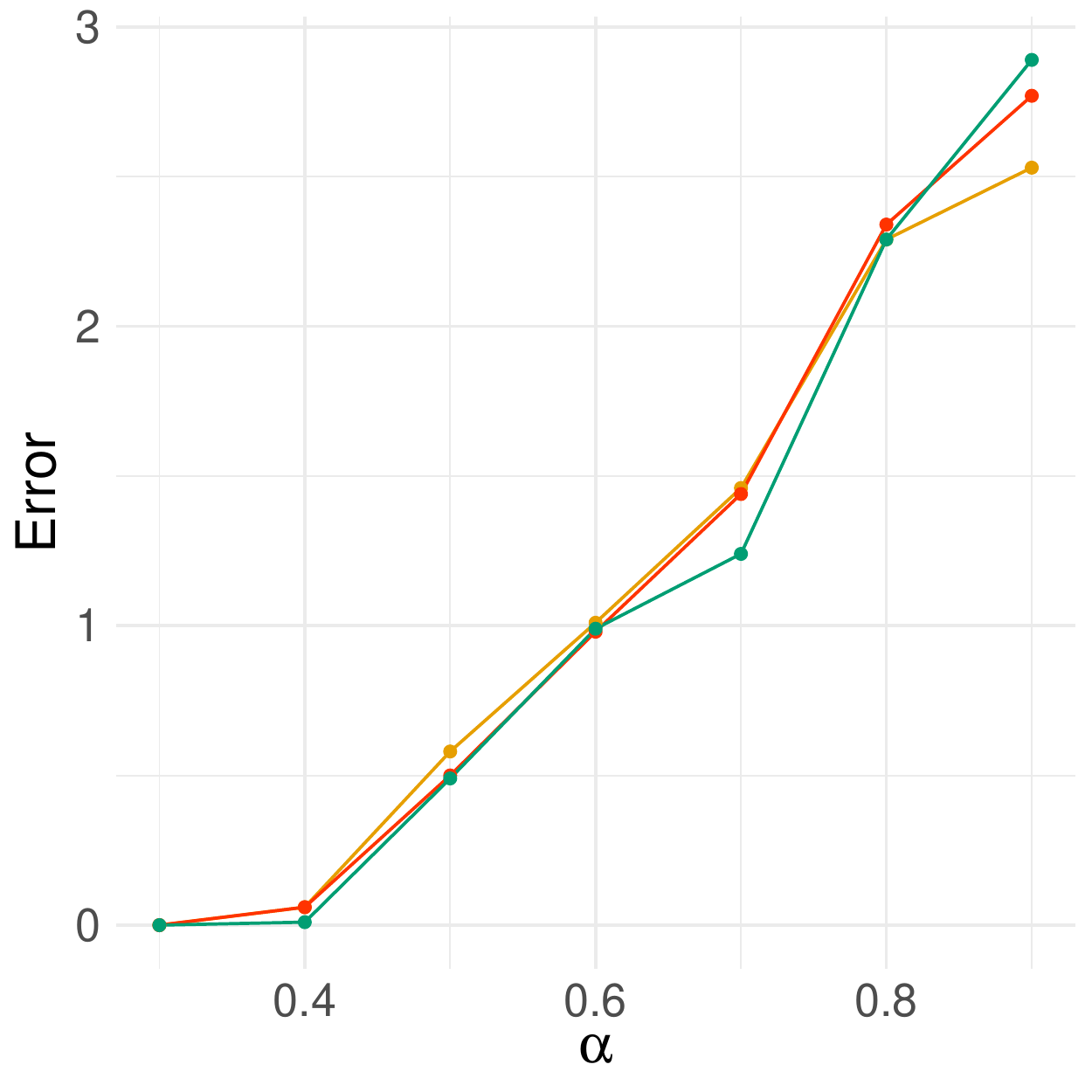} & \includegraphics[scale = 0.6]{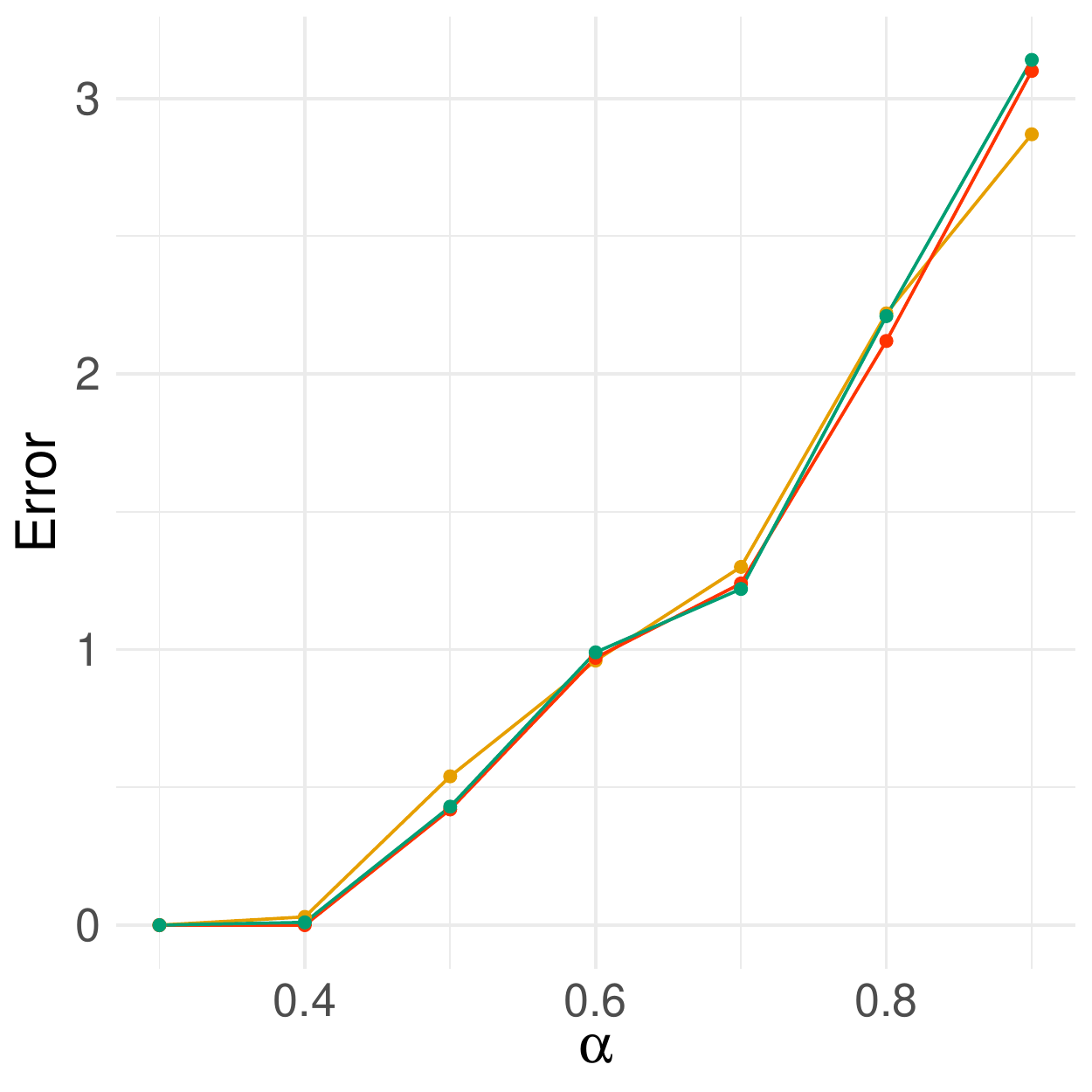}
    \end{tabular}
    \caption{Average absolute error in changepoint estimation for graphon model over 100 experiments. Left panel: $\Delta_p=0.05.$ Right panel: $\Delta_p=0.07.$ Parameters: $\Delta_e=100$ (yellow), 110 (red), 120 (green). We have used the CUSUM statistic with $\xi = 1/2$.}
    \label{fig:change_pt_est}
\end{figure}

\section{Proofs}
\label{sec:proof}
For any random object $X$, we will use the notation $\bkt{X}$ to denote $X-\E(X)$. Also, for matrices $A$ and $B$, we write $A\pre B$ to mean that $B-A$ is non-negative definite.  We would need the following lemma whose proof is provided at the end of this section.
\begin{lemma} \label{lem:A est} For the notation defined in \textsection\ref{sec:setup}, if we denote $A=((a_{ij}))_{i,j\in[n]}$ and $\bar{a}_{ij}=\E(a_{ij}),$ then the following holds.
\begin{enumerate}
    \item[(a)] $\bar a_{ij}=(T+1)p_{ij}$ for all $i, j\in[n]$.
    \item[(b)] For any $\alpha\in(0,1)$, if $\sqrt T>\log_{1/\alpha}(T)$, then 
    \[
        \var(a_{ij}) \le Tp_{ij}\frac{3-\alpha}{1-\alpha} \text{ for all } i, j\in[n].
    \]
    \item[(c)] Let $$F_{ij} := \bigg\{|a_{ij} - \bar a_{ij}| \le \sqrt{\frac{Td_{\max}}{\log(n)}}\bigg\}.$$ Then $\mathbb P(F_{ij}^c) \le 1/n^3$ for all $i, j\in[n]$.
    \item[(d)] If 
    \[
        \bar\ell_i:=\sum_{j\in[n]} \bar a_{ij} \text{ for } i\in[n], \text{ and } \widetilde F_{ij} := \bigg\{|a_{ij} - \bar a_{ij}| \le \sqrt{\frac{\bar\ell_i \bar\ell_j}{Td_{\min}\log(n)}}\bigg\},
    \]
    then $\mathbb P(\widetilde F_{ij}^c) \le 1/n^3$ for all $i, j\in[n]$.
    \item[(e)] If $$\ell_i:=\sum\limits_{j\in[n]} a_{ij}\text{ for }i\in[n],$$ then there is a constant $\kappa$ such that $$\mathbb P\bigg(\bigg|\sqrt{\frac{\ell_i}{\bar\ell_i}} -1\bigg|>\kappa\sqrt{\frac{\log(n/\delta)}{Td_{\min}}}\bigg)\le\frac{\delta}{2n}$$ for each $i\in[n]$.
\end{enumerate}
\end{lemma}

\begin{proof}[Proof of Theorem \ref{thm:main} (A)]
For $1\le i\le j\le n$, let $\mathbf{e}_{i}$ denote the $i$-th canonical basis vector of $\mathbb{R}^{n}.$ Define
$$E_{ij}:=\begin{cases}\mathbf{e}_{i} \mathbf{e}_{j}^{\top}+\mathbf{e}_{j} \mathbf{e}_{i}^{\top}&\text{if }i\ne j\\\mathbf{e}_{i} \mathbf{e}_{i}^{\top}&\text{if }i=j\end{cases}$$
and $X_{ij}:=\bkt{a_{ij}}\mathbf 1_{F_{ij}}E_{ij}$ where $F_{ij}$ is the same as in Lemma \ref{lem:A est}(c).
Clearly
$$\bkt{A}=\sum_{1\le i\le j\le n}\bkt{a_{ij}}E_{ij}\implies\bkt{A}=\sum_{1\le i\le j\le n} X_{ij}\text{ on the event }F:=\bigcap_{1\le i\le j\le n}F_{ij}.$$
Also, by the triangle inequality,
$$\bigg\|\sum_{1\le i\le j\le n} X_{ij}\bigg\| \le \bigg\|\sum_{1\le i\le j\le n} \bkt{X_{ij}}\bigg\|+\bigg\|\sum_{1\le i\le j\le n} \E (X_{ij})\bigg\|=:A_1+A_2.$$
To bound $A_2$, we use the fact that $\|\mathbf U\| \le \max_{i\in[n]}\sum_{j\in[n]}|u_{ij}|$ for any symmetric $n\times n$ matrix $\mathbf U$. Since $\E\bkt{a_{ij}}=0$, we have $\E(\bkt{a_{ij}}\mathbf 1_{F_{ij}})=\E(\bkt{a_{ij}}\mathbf 1_{F_{ij}^c}).$ Now can use Cauchy-Schwartz inequality and the bound for $\mathbb P(F_{ij}^c)$ given in Lemma \ref{lem:A est}(c) to have
\begin{align}
A_2 \le \max_{i\in[n]}\sum_{j\in[n]} |\E(\bkt{a_{ij}}\mathbf 1_{F_{ij}^c})|
&\le \max_{i\in[n]}\sum_{j\in[n]}\sqrt{\E\bkt{a_{ij}}^2\mathbb P(F_{ij}^c)} \notag \\
&\le \max_{i\in[n]}\sum_{j\in[n]}\frac{\sqrt{Tp_{ij}}}{n} \le \sqrt{\frac{Td_{\max}}{n}}. \label{A2 bd}
\end{align}
To bound $A_1$, we will use the standard matrix Bernstein inequality (see, e.g., Corollary 7.1 of \cite{oliveira2009concentration}). Since 
$\|E_{ij}\|=1$, so 
\begin{align}
\left\|\bkt{X_{ij}}\right\|=\left|\Bkt{\bkt{a_{ij}}\mathbf 1_{F_{ij}}} \right| = \left|\Bkt{\bkt{a_{ij}}\mathbf 1(|\bkt{a_{ij}}|\le(Td_{\max}/\log(n))^{1/2})} \right| \le 2\sqrt{\frac{Td_{\max}}{\log(n)}}. \label{X norm bd}
 \end{align}
Using Lemma \ref{lem:A est}(b) and noting that $E_{ij}^2=\mathbf e_i\mathbf e_i^\top+\mathbf e_j\mathbf e_j^\top$ (respectively $\mathbf e_i\mathbf e_i^\top$) for $i<j$ (respectively $i=j$), 
\begin{align*}
\sum_{1\leq i\leq j\leq n}\mathbb{E}\Bkt{X_{ij}}^2 &\preccurlyeq \sum_{1\leq i\leq j\leq n}\E\bkt{a_{ij}}^2 E_{ij}^2 \\
&\preccurlyeq\frac{(3-\alpha)}{1-\alpha}\sum_{1\le i<j\le n}Tp_{ij}\left(\mathbf{e}_{i} \mathbf{e}_{i}^{\top}+\mathbf{e}_{j} \mathbf{e}_{j}^{\top}\right) 
\preccurlyeq \frac{Td_{\max}(3-\alpha)}{(1-\alpha)}\mathbf I,
\end{align*}
which shows $\bigg\|\sum\limits_{1\leq i\leq j\leq n}\mathbb{E}\bkt{X_{ij}}^2\bigg\|\le Td_{\max}(3-\alpha)/(1-\alpha)$. Combining this with \eqref{X norm bd}, and applying the matrix Bernstein inequality, taking
$$t=C_1\sqrt{Td_{\max}\log(n/\delta)},\text{ where }\delta\in(n^{-c},\,1/2)$$
and $C_1=C_1(\alpha, c)$ satisfies $C_1^2/2\geq8(3-\alpha)/(1-\alpha)+4C_1\sqrt{c+1}$, we see that
\begin{align}
   \mathbb{P}\bigg(\bigg\|\sum_{1\leq i\leq j\leq n}\bkt{X_{ij}}\bigg\|\geq t\bigg) 
  & \leq 2n\exp{\left(-\frac{t^2}{\frac{8Td_{\max}(3-\alpha)}{(1-\alpha)}+4t\sqrt{\frac{Td_{\max}}{\log(n)}}}\right)} \nonumber\\
  & \leq 2n\exp\left(-\frac{ C_1^2Td_{\max}\log(n/\delta)}{\frac{8Td_{\max}(3-\alpha)}{(1-\alpha)}+4C_1\sqrt{c+1}Td_{\max}}\right) \le \delta. \label{1st_Bern}
\end{align}
Combining \eqref{1st_Bern} with \eqref{A2 bd},
$$\mathbb P\bigg(\|\bkt{A}\| > 2C_1\sqrt{Td_{\max}\log(n/\delta)}; F\bigg)\le\delta.$$
Using Lemma \ref{lem:A est}(c) and union bound, we get $\mathbb P(F^c)\sim1/n.$ Taking $\kappa_1=2C_1$ completes the proof of (A).
\end{proof}

\begin{proof}[Proof of Theorem \ref{thm:main} (B)]
Let $D$ (resp.~$\bar D$) be the diagonal matrix having $\ell_i^{-1/2}$ (resp.~$\bar\ell_i^{-1/2}$) at the $i$-th diagonal entry. Then, $\mathcal L=I-DAD$ and $\bar{\mathcal L}=I-\bar D\bar A\bar D$.
Define $\mathcal M:=I-\bar DA\bar D$. We will show that there are constants $C_1(\alpha),\,C_2(\alpha)$ such that 
\begin{align}
\mathbb P\left(\|\mathcal L-\mathcal M\|> C_1\sqrt{\frac{\log(n/\delta)}{Td_{\min}}}\right) \le \delta/2,  \label{Lapbd 1}\\
\mathbb P\left(\|\bar{\mathcal L}-\mathcal M\|> C_2\sqrt{\frac{\log(n/\delta)}{Td_{\min}}}; \widetilde F\right) \le \delta/2, \label{Lapbd 2}
\end{align}
where $\widetilde F=\bigcap\limits_{1\le i\le j\le n}\widetilde{F}_{ij}.$ The conclusion of Main Theorem (B) follows by combining the estimates in \eqref{Lapbd 1}, \eqref{Lapbd 2}, and Lemma \ref{lem:A est}(d), and using the triangle inequality. To show \eqref{Lapbd 1}, note that $I-\mathcal M=\bar DD^{-1}(I-\mathcal L)D^{-1}\bar D$, and hence using the triangle inequality and the fact that $\|I-\mathcal L\|\le 1$,
\begin{align}
    \left\|\mathcal{L} - \mathcal{M}\right\| =& \left\|\left(\bar{D}D^{-1}\right) \left(I-\mathcal{L}\right) \left(D^{-1}\bar{D}\right)-\left(I-\mathcal{L}\right)\right\|\nonumber \\
    \leq &\left\|\left(\bar{D}D^{-1}-I\right)\left(I-\mathcal{L}\right)\left(D^{-1}\bar{D}\right)\right\|+ \left\|\left(I-\mathcal{L}\right)\left(D^{-1}\bar{D}-I\right)\right\|\nonumber \\
     \leq &\left\|\bar{D}D^{-1} - I\right\| \left\| I - \mathcal{L}\right\| \left\|D^{-1}\bar{D}\right\|+\left\|I-\mathcal{L}\right\| \left\|D^{-1}\bar{D}-I\right\|\nonumber\\
     \leq &\left\|\bar{D}D^{-1} - I\right\| \left\|D^{-1}\bar{D}\right\|+ \left\|D^{-1}\bar{D}-I\right\|\nonumber\\
     \leq &\left\|\bar{D}D^{-1} - I\right\|(\left\|\bar{D}D^{-1} - I\right\|+2).\label{intm_bound}
\end{align}
To bound the RHS of \eqref{intm_bound}, note that $\|\bar DD^{-1}-I\|=\max\limits_{i\in[n]}|\sqrt{\ell_i/\bar\ell_i}-1|$. So, using union bound and Lemma \ref{lem:A est}(e),
\begin{align}
    \mathbb P\bigg(\|\bar DD^{-1}-I\|>\kappa\sqrt{\frac{\log(n/\delta)}{Td_{\min}}}\bigg)\le\delta/2,\label{intm_prob}
\end{align}
where $\kappa$ is as in Lemma \ref{lem:A est}(e). This observation \eqref{intm_prob} together with \eqref{intm_bound} and the fact that $Td_{\min}>(c+1)\log(n)$ proves \eqref{Lapbd 1} if we take $C_1=\kappa(\kappa+2)$. To show \eqref{Lapbd 2}, let $b_{ij}:=a_{ij}/\sqrt{\bar\ell_i\bar\ell_j}$ and $Y_{ij}:=\bkt{b_{ij}}\mathbf{1}_{\widetilde{F}_{ij}}E_{ij}$ for $i, j\in[n]$. Clearly,
$$\bar{\mathcal L}-\mathcal M=\sum_{1\le i \leq j\le n}Y_{ij}\text{ on the event }\widetilde{F}:=\bigcap_{1\le i\le j\le n}\widetilde{F}_{ij}.$$
Also, by the triangle inequality,
$$\bigg\|\sum_{1\le i\le j\le n} Y_{ij}\bigg\| \le \bigg\|\sum_{1\le i\le j\le n} \bkt{Y_{ij}}\bigg\|+\bigg\|\sum_{1\le i\le j\le n} \E (Y_{ij})\bigg\|=:B_1+B_2.$$
To bound $B_2$, we again use the fact that $\|\mathbf U\| \le \max_{i\in[n]}\sum_{j\in[n]}|u_{ij}|$ for any symmetric $n\times n$ matrix $\mathbf U$. Since $\E\bkt{b_{ij}}=0$, we have $\E(\bkt{b_{ij}\mathbf{1}_{\widetilde{F}_{ij}}})=\E(\bkt{b_{ij}\mathbf{1}_{\widetilde{F}_{ij}^c}}).$ Now we can use Cauchy-Schwartz inequality and the bound for $\mathbb P(\widetilde{F}_{ij}^c)$ given in Lemma \ref{lem:A est}(d) to have
\begin{align}
B_2 &\le \max_{i\in[n]}\sum_{j\in[n]} |\E(\bkt{b_{ij}}\mathbf 1_{\widetilde{F}_{ij}^c})|
\le \max_{i\in[n]}\sum_{j\in[n]}\sqrt{\E\bkt{b_{ij}}^2\mathbb P(\widetilde{F}_{ij}^c)} \notag \\
&\le \max_{i\in[n]}\sum_{j\in[n]} \sqrt{\frac{Tp_{ij}}{n^3\bar\ell_i\bar\ell_j}} 
\le \max_{i\in[n]}\sum_{j\in[n]} \frac{1}{\sqrt{n^3\bar\ell_i}} \le \frac{1}{\sqrt{nTd_{\min}}}. \label{B2 bd}
\end{align}
To bound $B_1$, we will use the matrix Bernstein inequality. Since 
$\|E_{ij}\|=1$, so 
\begin{align}
\|\bkt{Y_{ij}}\| = \left|\Bkt{\bkt{b_{ij}}\mathbf 1_{\widetilde F_{ij}}} \right| & \le \left|\Bkt{\bkt{b_{ij}}\mathbf 1(|\bkt{b_{ij}}|\le(Td_{\min}\log(n))^{-1/2})}\right| \nonumber \\
&\le \frac{2}{\sqrt{Td_{\min}\log(n)}}. \label{Y norm bd}
\end{align}
Using Lemma \ref{lem:A est}(b) and noting that $E_{ij}^2=\mathbf e_i\mathbf e_i^\top+\mathbf e_j\mathbf e_j^\top$ (resp.~$\mathbf e_i\mathbf e_i^\top$) for $i<j$ (resp.~$i=j$), 
\begin{align*}
\sum_{1\leq i\leq j\leq n}\mathbb{E}\Bkt{Y_{ij}}^2 \preccurlyeq \sum_{1\leq i\leq j\leq n}\E\bkt{b_{ij}}^2 E_{ij}^2
&\preccurlyeq\frac{(3-\alpha)}{(1-\alpha)}\sum_{1\le i<j\le n}\frac{Tp_{ij}}{\bar\ell_i\bar\ell_j}\left(\mathbf{e}_{i} \mathbf{e}_{i}^{\top}+\mathbf{e}_{j} \mathbf{e}_{j}^{\top}\right) \\
&\preccurlyeq \sum_{i\in[n]}\left(\frac{1}{\bar\ell_i}\sum_{j\in[n]}\frac{T(3-\alpha)p_{ij}}{(1-\alpha)\bar\ell_j}\right) \mathbf e_i\mathbf e_i^\top,
\end{align*}
which gives
\begin{align}
\bigg\|\sum_{1\leq i\leq j\leq n}\mathbb{E}\bkt{Y_{ij}}^2\bigg\|\le \max_{i\in[n]}\left(\frac{1}{\bar\ell_i}\sum_{j\in[n]}\frac{T(3-\alpha)p_{ij}}{(1-\alpha)\bar\ell_j}\right) \nonumber &\le \max_{i\in[n]}\left(\frac{1}{\bar\ell_i}\sum_{j\in[n]}\frac{(3-\alpha)p_{ij}}{(1-\alpha) d_{\min}}\right)\notag\\
&\le \frac{(3-\alpha)}{(1-\alpha) Td_{\min}}. \label{lap_Bern}
\end{align}
Combining \eqref{lap_Bern} with \eqref{Y norm bd}, applying the matrix Bernstein inequality, and taking
$$t=C_2\sqrt{\log(n/\delta)/Td_{\min}}\text{ where }\delta\in(n^{-c},\,1/2)$$
and $C_2=C_2(\alpha, c)$ satisfies $C_2^2/2\geq8(3-\alpha)/(1-\alpha)+8C_2\sqrt{c+1},$ we see that
\begin{align}
   \mathbb{P}\left(\bigg\|\sum_{1\leq i\leq j\leq n}\bkt{Y_{ij}}\bigg\|\geq t\right) 
  & \leq 2n\exp{\left(-\frac{t^2}{\frac{8(3-\alpha)}{(1-\alpha)Td_{\min}}+4t\sqrt{\frac{1}{Td_{\min}\log(n)}}}\right)} \notag \\
  & \leq 2n\exp\left(-\frac{C_2^2\frac{\log(4n/\delta)}{Td_{\min}}}{\frac{8(3-\alpha)}{Td_{\min}(1-\alpha)}+8C_2\frac{\sqrt{c+1}}{Td_{min}}}\right)\notag \\
  & \le \exp\left(-\frac{C_2^2\log(4n/\delta)}{\frac{8(3-\alpha)}{(1-\alpha)}+8C_2\sqrt{c+1}}\right)
  \le \delta/2. \label{lap_prob}
\end{align}
Combining \eqref{lap_prob} with \eqref{B2 bd} proves \eqref{Lapbd 2}. This completes the proof of (B).
\end{proof}

\begin{proof}[Proof of Lemma \ref{lem:A est}]
We will use the following construction of the lazy IER process based on the following independent sequences of random variables indexed by the set of all potential edges $\{ij: 1\le i\le j\le n\}$ of the networks.
$$\left(I_{ij}^{(t)}, t\ge 0\right) \;\sa{i.i.d.}\;\mathrm{Bernoulli}(p_{ij}),\text{ and }\left(G_{ij}^{(k)}, k\in\mathbb N\right) \;\sa{i.i.d.}\; \mathrm{Geometric}(1-\alpha).$$
For $1\le i\le j \le n$ and $k\ge 1$, let
$$S_{ij}^{(0)} := 0,\,S_{ij}^{(k)} := \sum_{\ell\in[k]} G_{ij}^{(\ell)} \text{ for } k\in\mathbb N,\text{ and } T_{ij} := \max\left\{k\ge 0: S_{ij}^{(k)} \le T\right\},$$
so $T_{ij} \sim\mathrm{Bin}(T, 1-\alpha)$. Then, for the edge $ij$,  $S_{ij}^{(k)}$ $\left(\text{respectively } T_{ij},\, I^{(t)}_{ij}\right)$ has the same distribution as the $k$-th renewal time (respectively the number of renewals by time $T$, the indicator of whether it is present at time $t$ when its status is renewed at that time). So, if $A=((a_{ij}))_{i,j\in[n]}$, then 
\[
    a_{i j}\eq{d}\bigg(\sum_{k=1}^{T_{ij}} G_{ij}^{(k)}I_{i j}^{\left(S_{i j}^{(k-1)}\right)}\bigg) + \big(T + 1 - S_{ij}^{(T_{ij})}\big) I_{ij}^{\big(S_{i j}^{(T_{ij})}\big)}
\]
for all $i, j \in[n].$ Using the above representation, we will prove Lemma \ref{lem:A est}.
\vskip10pt
\noindent
{\bf (a).} Let $\mathcal G_{ij}:=\left(G_{ij}^{(t)}, t\ge 0\right)$. Observe that
\begin{align*}
\mathbb{E}\left(a_{i j} \mid \mathcal G_{ij}\right) = \sum_{k=1}^{T_{ij}} G_{i j}^{(k)} p_{i j}+\left(T+1-S_{i j}^{(T_{ij})}\right) p_{i j}=(T+1)p_{i j}.
\end{align*}
Hence $\bar a_{i j}=(T+1)p_{i j}$.
\vskip10pt
\noindent
{\bf (b).}  Here, $\mathbb{V}\text{ar}(a_{i j})=\mathbb{E}(\mathbb{V}\text{ar}(a_{i j} \mid \mathcal G_{i j}))$, as $\mathbb{E}(a_{i j} \mid \mathcal G_{i j})$ is a constant. So,
\begin{align}
    \mathbb{V}\text{ar}(a_{i j}) &= \mathbb{E}(\mathbb{V}\text{ar}(a_{i j} \mid \mathcal G_{i j})) \nonumber \\
&=p_{i j}(1-p_{i j})\mathbb{E}\bigg[\sum_{k\in[T_{ij}]} \left(G_{i j}^{(k)}\right)^2 +\left(T+1-S_{i j}^{(T_{ij})}\right)^{2} \bigg].\label{eq:var}
\end{align}
Since $T_{ij}$ is a stopping time, using Wald's first equation the first sum in (\ref{eq:var}) is
\begin{align}
    \mathbb E\bigg[\sum_{k\in[T_{ij}]} \left(G_{i j}^{(k)}\right)^2\bigg]  &= \mathbb E (T_{ij})\cdot \mathbb E \left(G_{i j}^{(1)}\right)^2= \frac{T(1-\alpha)(1+\alpha)}{(1-\alpha)^{2}}.\label{eq:1st}
\end{align}
On the other hand, if $B:=\left\{S_{i j}^{(T_{ij})}>T+1-\sqrt T\right\}$, then $\mathbb P(B^c)\le\alpha^{\sqrt T}$, and hence
\begin{align}
\mathbb E\left[\left(T+1-S_{i j}^{(T_{ij})}\right)^{2} \mathbf 1_B\right] &\le T \text{ and }\notag\\
\mathbb E\left[\left(T+1-S_{i j}^{(T_{ij})}\right)^{2} \mathbf 1_{B^c}\right]& \le T^2\mathbb P(B^c)\le T^2\alpha^{\sqrt T} \le T\label{eq:2nd}
\end{align}
if $T$ is large enough. Combining (\ref{eq:var}), (\ref{eq:1st}) and (\ref{eq:2nd}), we see that there is a constant $c_1(\alpha)$ such that
\begin{align}
\mathbb{V}\text{ar}(a_{i j})  \le Tp_{ij}(1-p_{ij})\frac{3-\alpha}{1-\alpha} \text{ if $T\ge c_1(\alpha)$}.\label{eq:var_aij}
\end{align}

\noindent
{\bf (c).} The argument for bounding $\mathbb P(F_{ij}^c)$ consists of two cases.
\vskip5pt
\noindent
\boxed{\textbf{Case 1.}}
$$Td_{\max}\ge \frac{(1-\alpha)^2}{9(1+\alpha)^2}\frac{n^2}{\log(n)}.$$
Let $H_{ij}:= \{T_{ij} < (1+x)T(1-\alpha)\}$, where
$$x=3\sqrt{\frac{n\log(n)}{Td_{\max}(1-\alpha)}}\le\frac{6}{(1-\alpha)^{3/2}}\frac{\log(n)}{\sqrt n}\le 1 \text{ if $n$ is large enough.}$$
Using a standard binomial upper-tail large deviation estimate (see, e.g., Theorem 4.4 of \cite{mitzenmacher2017probability}), we get
\begin{align}
\mathbb P(H_{ij}^c) 
&\le \exp\left(-T(1-\alpha)\frac{x^2}{2+x}\right) 
\le e^{-T(1-\alpha)x^2/3} \notag \\
&=  \exp\left(-T(1-\alpha) \frac{9(\log(n))^2}{3n(1-\alpha)}\frac{Td_{\max}/\log(n)}{(Td_{\max}/n)^2}\right) 
\le e^{-3\log(n)} \le 1/n^3. \label{H_est}
\end{align}
$a_{ij}$ is stochastically dominated  by $\sum_{k\in[(1+x) T(1-\alpha)]}I_{ij}^{(k)}G_{ij}^{(k)}$ on $H_{ij}$, so for any $0<\lambda<\log\frac 1\alpha$,
\begin{align*}
\mathbb E[e^{\lambda a_{ij}}\mathbf 1_{H_{ij}}] &\le \mathbb E\exp\left[\lambda \sum_{k\in[(1+x) T(1-\alpha)]}I_{ij}^{(k)}G_{ij}^{(k)}\right]\\
&=\left[(1-p_{ij})+\frac{p_{ij}(1-\alpha)}{e^{-\lambda}-\alpha }\right]^{(1+x) T(1-\alpha)} \\
&\le \exp\left((1+x)\bar a_{ij}\psi(\lambda)\right),
\end{align*}
where
$$\psi(\lambda):=\frac{(1-\alpha)^2}{e^{-\lambda}-\alpha}-(1-\alpha)\implies\psi'(\lambda)=\frac{(1-\alpha)^2e^{-\lambda}}{(e^{-\lambda}-\alpha)^2}.$$
Also, $\psi(0)=0$ and $\psi''(\lambda)\ge 0$ for all $\lambda$, so $\psi(\lambda)\le \lambda\psi'(\lambda)$ for all $\lambda$ using a standard analysis argument. Using this, letting
$$\lambda:=\frac{6(\log(n))^{3/2}}{\sqrt{Td_{\max}}},$$
and applying Markov inequality for $\mathbb P(\cdot|H)$,  
\begin{align} \nonumber
\mathbb P\bigg(a_{ij}>\bar a_{ij} &+\sqrt{\frac{Td_{\max}}{\log(n)}}; H\bigg) \\ 
&\le \exp\left(\lambda \bar a_{ij} \left\{(1+x) \psi'(\lambda) -1-\sqrt{\frac{Td_{\max}}{\log(n)}}\frac{1}{\bar a_{ij}}\right\} \right). \label{eq:uptail}
\end{align}
To estimate the RHS of \eqref{eq:uptail} note that if $n$ is large enough, then
\[
\lambda = \frac{6(\log(n))^2}{n}\frac{\sqrt{Td_{\max}/\log(n)}}{Td_{\max}/n}  
\le \frac{18(1+\alpha)(\log(n))^2}{n(1-\alpha)} <\frac{1-\alpha}{2},
\]
so that
\begin{equation}
\psi'(\lambda) \le \left(\frac{1-\alpha}{1-\lambda-\alpha}\right)^2
\le \left(1+\frac{2\lambda}{1-\alpha}\right)^2
\le \left(1+\frac{6\lambda}{1-\alpha}\right), \label{psi'bd}
\end{equation}
and
\begin{align*}
(1+x)\psi'(\lambda) &\le 1+x+\frac{12\lambda}{1-\alpha} \\ 
&\le 1+\left[\frac{6\log(n)}{\sqrt{n(1-\alpha)}}+\frac{12(\log(n))^2}{n(1-\alpha)}\right]\frac{\sqrt{Td_{\max}/\log(n)}}{Td_{\max}/n} \\
&\le 1+\frac 12 \frac{\sqrt{Td_{\max}/\log(n)}}{Td_{\max}/n} \le 1+\frac 12 \frac{\sqrt{Td_{\max}/\log(n)}}{\bar a_{ij}}.
\end{align*}
Plugging this bound in the RHS of \eqref{eq:uptail}, 
\begin{align}
\mathbb P\left(a_{ij}>\bar a_{ij}+\sqrt{\frac{Td_{\max}}{\log(n)}}; H_{ij}\right) 
\le \exp\left(-\frac\lambda 2 \sqrt{Td_{\max}/\log(n)}\right)
\le n^{-3}. \label{UpperTail}
\end{align}
Now, we estimate the lower tail of $a_{ij}$. Let $\widetilde H_{ij}:=\{T_{ij}>(1-x)T(1-\alpha)
\}$. Then using a standard binomial lower-tail large deviation estimate (see, e.g., Theorem 4.5 of \cite{mitzenmacher2017probability}) and 
the lower bound for $x^2$ used in \eqref{H_est},
\begin{align}
\mathbb P(\widetilde H_{ij}^c) \le e^{-T(1-\alpha)x^2/4} \le n^{-3}.\label{H_tilde_est} 
\end{align}
Clearly, $a_{ij}$ stochastically dominats $\sum_{k\in[(1-x)(1-\alpha)T]}I_{ij}^{(k)}G_{ij}^{(k)}$ on $\widetilde H_{ij}$, so for any $\lambda>0$,
\begin{align*}
    \mathbb E[e^{-\lambda a_{ij}}\mathbf 1_{\widetilde H_{ij}}] &\le \mathbb E\exp\bigg[-\lambda \sum_{k\in[(1-x) T(1-\alpha)]}I_{ij}^{(k)}G_{ij}^{(k)}\bigg] \\
                                                                &=\bigg[(1-p_{ij})+\frac{p_{ij}(1-\alpha)}{e^{\lambda}-\alpha }\bigg]^{(1-x) T(1-\alpha)} \le \exp\left(-\bar a_{ij}(1-x)\phi(\lambda)\right),
\end{align*}
where $\phi(\lambda):=(1-\alpha)-\frac{(1-\alpha)^2}{e^{\lambda}-\alpha}$ satisfies $\phi'(\lambda)=\frac{(1-\alpha)^2e^{\lambda}}{(e^{\lambda}-\alpha)^2}$.
Also, $\phi(0)=0$ and $\phi''(\lambda)\le 0$ for all $\lambda$, so $\phi(\lambda)\ge \lambda\phi'(\lambda)$ for all $\lambda$, using a standard analysis argument. Using this, recalling the value of $\lambda$, and applying Markov inequality for $\mathbb P(\cdot|\widetilde H_{ij})$,  
\begin{align}
\mathbb P\bigg(a_{ij}<\bar a_{ij} -\sqrt{\frac{Td_{\max}}{\log(n)}}; \widetilde H_{ij}\bigg) \le \exp\left(\lambda \bar a_{ij} \left\{1-\sqrt{\frac{Td_{\max}}{\log(n)}}\frac{1}{\bar a_{ij}}-(1-x) \phi'(\lambda)\right\} \right). \label{eq:lowtail}
\end{align}
To estimate the RHS of \eqref{eq:lowtail} note that
\begin{align*}
    &\frac{e^\lambda-1}{e^\lambda-\alpha}=\frac{1-e^{-\lambda}}{1-\alpha e^{-\lambda}}\le\frac{1-e^{-\lambda}}{1-\alpha}\le \frac{\lambda}{1-\alpha}\\
    \implies & 1-\frac{e^\lambda-1}{e^\lambda-\alpha}\ge1-\frac{\lambda}{1-\alpha}\implies\frac{1-\alpha}{e^\lambda-\alpha}\ge1-\frac{\lambda}{1-\alpha}.
\end{align*}
Now, take $\lambda = \frac{6(\log(n))^2}{n}\frac{\sqrt{Td_{\max}/\log(n)}}{Td_{\max}/n}$ to obtain
\begin{equation}
\phi'(\lambda)\ge \left(\frac{1-\alpha}{e^\lambda-\alpha}\right)^2
= \left(1-\frac{e^\lambda-1}{e^\lambda-\alpha}\right)^2
\ge \left(1-\frac{\lambda}{(1-\alpha)}\right)^2
\ge \left(1-\frac{2\lambda}{1-\alpha}\right), \label{phi'bd}
\end{equation}
and
\begin{align*}
(1-x)\phi'(\lambda) &\ge 1-x-\frac{\lambda}{1-\alpha} \\
&\ge 1-\left[\frac{6\log(n)}{\sqrt{n(1-\alpha)}}+\frac{(\log(n))^2}{n(1-\alpha)}\right]\frac{\sqrt{Td_{\max}/\log(n)}}{Td_{\max}/n} \\
&\ge 1-\frac 12 \frac{\sqrt{Td_{\max}/\log(n)}}{Td_{\max}/n} \ge 1-\frac 12 \frac{\sqrt{Td_{\max}/\log(n)}}{\bar a_{ij}}.
\end{align*}
Plugging this bound in the RHS of \eqref{eq:lowtail}, we have
\begin{align}
\mathbb P\left(a_{ij}<\bar a_{ij}-\sqrt{\frac{Td_{\max}}{\log(n)}}; \widetilde H_{ij}\right) 
\le \exp\left(-\frac\lambda 2 \sqrt{Td_{\max}/\log(n)}\right)
\le n^{-3}. \label{LowerTail}
\end{align}
Combining \eqref{H_est}, \eqref{UpperTail}, \eqref{H_tilde_est}, \eqref{LowerTail}, we get the desired bound for $\mathbb P(F_{ij}^c)$ for Case 1.
\vskip5pt
\noindent
\boxed{\textbf{Case 2.}}
$$Td_{\max}<\frac{(1-\alpha)^2}{9(1+\alpha)^2}\frac{n^2}{\log(n)}.$$
Since $Td_{\max}/n\le\frac{1-\alpha}{3(1+\alpha)}\sqrt{Td_{\max}/\log(n)}$ in Case 2, the lower tail probability (LHS of \eqref{LowerTail}) of $a_{ij}$ must be 0, as $a_{ij}$ is nonnegative. To estimate the upper tail of $a_{ij}$, note that $a_{ij}$ is stochastically dominated  by $\sum_{k\in[T]}I_{ij}^{(k)}G_{ij}^{(k)}$. So, for any $0<\lambda<\log(1/\alpha)$,
\begin{align*}
\mathbb E[e^{\lambda a_{ij}}] \le \mathbb E\exp\left[\lambda \sum_{k\in[T]}I_{ij}^{(k)}G_{ij}^{(k)}\right] = \left[(1-p_{ij})+\frac{p_{ij}(1-\alpha)}{e^{-\lambda}-\alpha }\right]^{T} 
\le \exp\left( \bar a_{ij} \varphi(\lambda) \right),
\end{align*}
where $\varphi(\lambda):=\frac{1-\alpha}{e^{-\lambda}-\alpha}-1$ satisfies $\varphi'(\lambda)=\frac{(1-\alpha)e^{-\lambda}}{(e^{-\lambda}-\alpha)^2}$. Also, $\varphi(0)=0$ and $\varphi''(\lambda)\ge 0$ for all $\lambda$, so $\varphi(\lambda)\le \lambda\varphi'(\lambda)$ for all $\lambda$.
Using this, taking $\lambda=\log\frac{2}{1+\alpha}$, noting that $\varphi'(\lambda)=\frac{2(1+\alpha)}{1-\alpha} \le \frac 23 \sqrt{Td_{\max}/\log(n)}$, and applying Markov inequality, we get
\begin{align}
\mathbb P\left(a_{ij}>\bar a_{ij}+\sqrt{\frac{Td_{\max}}{\log(n)}}\right) 
&\le\exp\left(\bar a_{ij}\varphi(\lambda) -\lambda \bar a_{ij}-\lambda\sqrt{Td_{\max}/\log(n)} \right) \label{eq:uptail_case2} \\
&\le
\exp\left(\lambda \bar a_{ij} \left\{\frac{\varphi(\lambda)}{\lambda} -1-\frac{\sqrt{Td_{\max}/\log(n)}}{\bar a_{ij}}\right\} \right) \notag \\
&\le \exp\left(-\frac\lambda 3\sqrt{Td_{\max}/\log(n)}\right) \le n^{-3},   \notag
\end{align}
 if $Td_{\max}\ge(9/\lambda)^2(\log(n))^3$. The RHS of \eqref{eq:uptail_case2} gives the desired bound for $\mathbb P(F_{ij}^c)$ in Case 2.

\noindent
{\bf (d).} The argument for bounding $\mathbb P(\widetilde F_{ij}^c)$ consists of two cases.
\vskip5pt
\noindent
\boxed{\textbf{Case 1.}}
$$\bar a_{ij}\ge\frac{(1-\alpha)}{3(1+\alpha)}\sqrt{\bar\ell_i\bar\ell_j/(Td_{\min}\log(n))}.$$
Let $H_{ij}:= \{T_{ij} < (1+x)T(1-\alpha)\}$, where
\[
    x=\frac{1-\alpha}{3(1+\alpha)} \sqrt{\frac{\bar\ell_i\bar\ell_j}{Td_{\min}\log(n)}} \frac{1}{\bar a_{ij}} \le 1.
\]
Using the standard binomial upper-tail large deviation estimate, if $\frac{d_{\min}}{d_{\max}} \ge \frac{81(1+\alpha)^2}{(1-\alpha)^3}\frac{(\log(n))^2}{n}$, then
\begin{equation}\label{H_est2}
    \mathbb P(H_{ij}^c) \le \exp\left(-T(1-\alpha)\frac{x^2}{2+x}\right) \le e^{-T(1-\alpha)x^2/3} \le n^{-3},
\end{equation}
as
\[
    \frac{T(1-\alpha)x^2}{9\log(n)} \ge \frac{(1-\alpha)^3}{81(1+\alpha)^2}\frac{T\bar\ell_i\bar\ell_j}{Td_{\min}\log^2(n)T^2p_{ij}^2} \ge  \frac{(1-\alpha)^3}{81(1+\alpha)^2}\frac{nd_{\min}}{\log^2(n)d_{\max}} \ge 1.
\]
Now $a_{ij}$ is stochastically dominated  by $\sum_{k\in[(1+x) T(1-\alpha)]}I_{ij}^{(k)}G_{ij}^{(k)}$ on $H_{ij}$, so for any $0<\lambda<\log\frac 1\alpha$,
\begin{align*}
    \mathbb E[e^{\lambda a_{ij}}\mathbf 1_{H_{ij}}] \le \mathbb E\exp\left[\lambda \sum_{k\in[(1+x) T(1-\alpha)]}I_{ij}^{(k)}G_{ij}^{(k)}\right] &=\left[(1-p_{ij})+\frac{p_{ij}(1-\alpha)}{e^{-\lambda}-\alpha }\right]^{(1+x) T(1-\alpha)} \\
&\le \exp\left((1+x)\bar a_{ij}\psi(\lambda)\right),
\end{align*}
where  $\psi(\lambda):=\frac{(1-\alpha)^2}{e^{-\lambda}-\alpha}-(1-\alpha)$ satisfies $\psi'(\lambda)=\frac{(1-\alpha)^2e^{-\lambda}}{(e^{-\lambda}-\alpha)^2}$. Also, $\psi(0)=0$ and $\psi''(\lambda)\ge 0$ for all $\lambda$, so $\psi(\lambda)\le \lambda\psi'(\lambda)$ for all $\lambda$, using a standard analysis argument. Using this fact and the Markov inequality for $\mathbb P(\cdot|H_{ij})$, we get
\begin{align} \nonumber
\mathbb P\bigg(a_{ij}>\bar a_{ij}&+\sqrt{\frac{\bar\ell_i\bar\ell_j}{Td_{\min}\log(n)}}; H_{ij}\bigg) \\
&\le \exp\left(\lambda \bar a_{ij} \left\{(1+x) \psi'(\lambda) -1-\sqrt{\frac{\bar\ell_i\bar\ell_j}{Td_{\min}\log(n)}}\frac{1}{\bar a_{ij}}\right\} \right) \label{eq:uptail2}
\end{align}
for any $0<\lambda<\log(1/\alpha)$. Next we note that if
\[
    \lambda=\frac{(1-\alpha)^2}{36(1+\alpha)} \sqrt{\frac{\bar\ell_i\bar\ell_j}{Td_{\min}\log(n)}} \frac{1}{\bar a_{ij}},
\]
then
\[
    \lambda \le \frac{(1-\alpha)^2}{36(1+\alpha)} \frac{3(1+\alpha)}{1-\alpha}  <\frac{1-\alpha}{2} <\log(1/\alpha),
\]
as $x\le 1$. So
\[
    \psi'(\lambda) \le \left(\frac{1-\alpha}{1-\lambda-\alpha}\right)^2 \le \left(1+\frac{2\lambda}{1-\alpha}\right)^2 \le \left(1+\frac{6\lambda}{1-\alpha}\right),
\]
and
\begin{align*}
    (1+x)\psi'(\lambda) &\le 1+x+\frac{12\lambda}{1-\alpha} \\
                        &\le 1+\left[\frac{1-\alpha}{3(1+\alpha)}+\frac{1-\alpha}{3(1+\alpha)}\right] \sqrt{\frac{\bar\ell_i\bar\ell_j}{Td_{\min}\log(n)}}\frac{1}{\bar a_{ij}} \\ 
                        &\le 1+\frac 23 \sqrt{\frac{\bar\ell_i\bar\ell_j}{Td_{\min}\log(n)}}\frac{1}{\bar a_{ij}}. 
\end{align*}
Plugging this bound in the RHS of \eqref{eq:uptail2}, 
\begin{align}
\mathbb P\left(a_{ij}>\bar a_{ij}+\sqrt{\frac{\bar\ell_i\bar\ell_j}{Td_{\min}\log(n)}}; H_{ij}\right) 
&\le \exp\left(-\frac\lambda 3 \sqrt{\frac{\bar\ell_i\bar\ell_j}{Td_{\min}\log(n)}}\right) \notag\\
& \le \exp\left(-\frac{(1-\alpha)^2}{108(1+\alpha)}\frac{\bar\ell_i\bar\ell_j}{Td_{\min}\log(n)\bar a_{ij}}\right) \label{UpperTail2} \le n^{-3}
\end{align}
if $d_{\min}/d_{\max} \ge 324(1+\alpha)(1-\alpha)^{-2}(\log(n))^2/n$, since $\bar\ell_i, \bar\ell_j\ge Td_{\min}$ and $\bar a_{ij} \le Td_{\max}/n$. 

Now, we estimate the lower tail of $a_{ij}$. Let $\widetilde H_{ij}:=\{T_{ij}>(1-x)T(1-\alpha)
\}$. Then using the standard binomial lower-tail large deviation estimate and 
the lower bound for $x^2$ used in \eqref{H_est},
\begin{align}
\mathbb P(\widetilde H_{ij}^c) \le e^{-T(1-\alpha)x^2/4} \le n^{-3}.\label{H_tilde_est2} 
\end{align}
Clearly, $a_{ij}$ stochastically dominates $\sum_{k\in[(1-x)(1-\alpha)T]}I_{ij}^{(k)}G_{ij}^{(k)}$ on $\widetilde H_{ij}$, so for any $\lambda>0$,
\begin{align*}
    \mathbb E[e^{-\lambda a_{ij}}\mathbf 1_{\widetilde H_{ij}}] &\le \mathbb E\exp\bigg[-\lambda \sum_{k\in[(1-x) T(1-\alpha)]}I_{ij}^{(k)}G_{ij}^{(k)}\bigg] \\
                                                                &= \bigg[(1-p_{ij})+\frac{p_{ij}(1-\alpha)}{e^{\lambda}-\alpha }\bigg]^{(1-x) T(1-\alpha)} \le \exp(-\bar a_{ij}(1-x)\phi(\lambda)),
\end{align*}
where $\phi(\lambda):=(1-\alpha)-\frac{(1-\alpha)^2}{e^{\lambda}-\alpha}$ satisfies $\phi'(\lambda)=\frac{(1-\alpha)^2e^{\lambda}}{(e^{\lambda}-\alpha)^2}$. Also, $\phi(0)=0$ and $\phi''(\lambda)\le 0$ for all $\lambda$, so $\phi(\lambda)\ge \lambda\phi'(\lambda)$ for all $\lambda$. using a standard analysis argument. Using this, recalling the value of  $\lambda$, and applying Markov inequality for $\mathbb P(\cdot|\widetilde H_{ij})$,  
\begin{align} \nonumber
\mathbb P\bigg(a_{ij}<\bar a_{ij}&-\sqrt{\frac{\bar\ell_i\bar\ell_j}{Td_{\min}\log(n)}}; \widetilde H_{ij}\bigg) \\
&\le \exp\left(\lambda \bar a_{ij} \left\{1-\sqrt{\frac{\bar\ell_i\bar\ell_j}{Td_{\min}\log(n)}}\frac{1}{\bar a_{ij}}-(1-x) \phi'(\lambda)\right\} \right). \label{eq:lowtail2}
\end{align}
To estimate the RHS of \eqref{eq:lowtail2} first note that $\lambda\le\frac{1-\alpha}{12}<1/12$, so $e^\lambda<1+2\lambda$. Using this bound,
\[
    \phi'(\lambda) \ge \left(1-\frac{e^\lambda-1}{e^\lambda-\alpha}\right)^2
\ge \left(1-\frac{2\lambda}{1-\alpha}\right)^2 \ge \left(1-\frac{4\lambda}{1-\alpha}\right),
\]
and
\begin{align*}
    (1-x)\phi'(\lambda) &\ge 1-x-\frac{4\lambda}{1-\alpha} \\
                        &\ge 1-\left[\frac{1-\alpha}{3(1+\alpha)}+\frac{1-\alpha}{9(1+\alpha)}\right] \sqrt{\frac{\bar\ell_i\bar\ell_j}{Td_{\min}\log(n)}}\frac{1}{\bar a_{ij}} \\  
                        &\ge 1-\frac 23 \sqrt{\frac{\bar\ell_i\bar\ell_j}{Td_{\min}\log(n)}}\frac{1}{\bar a_{ij}}. 
\end{align*}
Plugging this bound in the RHS of \eqref{eq:lowtail2}, and following the argument that leads to \eqref{UpperTail2},
\begin{align}
\mathbb P\left(a_{ij}<\bar a_{ij}-\sqrt{\frac{\bar\ell_i\bar\ell_j}{Td_{\min}\log(n)}}; \widetilde H_{ij}\right) 
\le \exp\left(-\frac\lambda 2 \sqrt{\frac{\bar\ell_i\bar\ell_j}{Td_{\min}\log(n)}}\right)
\le n^{-3}. \label{LowerTail2}
\end{align}
Combining \eqref{H_est2}, \eqref{UpperTail2}, \eqref{H_tilde_est2}, \eqref{LowerTail2}, we get the desired bound for $\mathbb P(\widetilde F_{ij}^c)$ for Case 1.
\vskip5pt
\noindent
\boxed{\textbf{Case 2.}}
$$\bar a_{ij}<\frac{1-\alpha}{3(1+\alpha)}\sqrt{\bar\ell_i\bar\ell_j/(Td_{\min}\log(n))}.$$
So, in this case, the lower tail probability (LHS of \eqref{LowerTail2}) of $a_{ij}$ must be 0, as $a_{ij}$ is nonnegative. To estimate the upper tail of $a_{ij}$, note that $a_{ij}$ is stochastically dominated  by $\sum_{k\in[T]}I_{ij}^{(k)}G_{ij}^{(k)}$. So, for any $0<\lambda<\log(1/\alpha)$,
\begin{align*}
\mathbb E[e^{\lambda a_{ij}}] \le \mathbb E\exp\left[\lambda \sum_{k\in[T]}I_{ij}^{(k)}G_{ij}^{(k)}\right] 
&= \left[(1-p_{ij})+\frac{p_{ij}(1-\alpha)}{e^{-\lambda}-\alpha }\right]^{T} \\
&\le  \exp\left( \bar a_{ij} \varphi(\lambda) \right),
\end{align*}
where $\varphi(\lambda):=\frac{1-\alpha}{e^{-\lambda}-\alpha}-1$ satisfies $\varphi'(\lambda)=\frac{(1-\alpha)e^{-\lambda}}{(e^{-\lambda}-\alpha)^2}$. Also, $\varphi(0)=0$ and $\varphi''(\lambda)\ge 0$ for all $\lambda$, so $\varphi(\lambda)\le \lambda\varphi'(\lambda)$ for all $\lambda$.
Using this, taking $\lambda=\log\frac{2}{1+\alpha}$, applying Markov inequality, and noting that 
\begin{equation}
    \varphi'(\lambda)=\frac{2(1+\alpha)}{1-\alpha} \le \frac 23 \sqrt{\frac{\bar\ell_i\bar\ell_j}{Td_{\min}\log(n)}} \frac{1}{\bar a_{ij}},
\end{equation}
we have
\begin{align} \label{eq:uptail_case22} \nonumber
\mathbb P\left(a_{ij}>\bar a_{ij}+\sqrt{\frac{\bar\ell_i\bar\ell_j}{Td_{\min}\log(n)}}\right) 
&\le\exp\left(\bar a_{ij}\varphi(\lambda) -\lambda \bar a_{ij}-\lambda\sqrt{\frac{\bar\ell_i\bar\ell_j}{Td_{\min}\log(n)}} \right) \\ \nonumber
&\le
\exp\left(\lambda \bar a_{ij} \left\{\frac{\varphi(\lambda)}{\lambda} -1-\sqrt{\frac{\bar\ell_i\bar\ell_j}{Td_{\min}\log(n)}} \frac{1}{\bar a_{ij}}\right\} \right) \\
&\le \exp\left(-\frac\lambda 3\sqrt{\frac{\bar\ell_i\bar\ell_j}{Td_{\min}\log(n)}} \right) \le n^{-3},
\end{align}
 if $Td_{\min}\ge(9/\lambda)^2(\log(n))^3$. The RHS of \eqref{eq:uptail_case22} gives the desired bound for $\mathbb P(\widetilde F_{ij}^c)$ in Case 2.

\noindent
{\bf (e).} First we note that $|\sqrt{\ell_i/\bar\ell_i}-1| \le |\ell_i/\bar\ell_i-1|$ on $\{|\ell_i/\bar\ell_i-1|\le\frac 34\}$, since the map $x\mapsto\sqrt{1+x}$ satisfies
\begin{align*}
    \sqrt{1+x}-1=\frac{x}{2\sqrt{1+x^*}} \text{ for some } x^*\in[-x, x],
\end{align*}
we have $|\sqrt{1+x}-1|\leq |x|$ for any $x \in [-3/4,3/4]$. So it suffices to show that there is a $\kappa(\alpha)>0$  for which  $\mathbb P(|\ell_i/\bar\ell_i-1|>\kappa(\alpha)\sqrt{\log(n/\delta)/Td_{\min}}) \le \delta/2n$. To prove this bound, we define the events $L_i:=\cup_{\star\in\{+,-\}, j\in[n]} K_{ij}^\star \cup_{\star\in\{+,-\}} J_i^\star$, where 
\begin{equation}
    J_i^\pm:=\bigg\{\sum_{j\in[n]}\pm\bkt{T_{ij}p_{ij}} > \sqrt{3\bar\ell_i(1-\alpha)\log(n/\delta)}\bigg\},
\end{equation}
and
\begin{equation}
    K_{ij}^\pm:=\bigg\{\pm\bkt{T_{ij}} >\frac{4-3\alpha}{1-\alpha} \sqrt{T(1-\alpha)\log(n/\delta)}\bigg\}
\end{equation}
for  $i, j \in [n]$, and show 
\begin{align} \label{ingrad}
    &(i)\,\, \mathbb P(L_i) \le \delta/(6n), \\ \nonumber
    &(ii)\,\, \mathbb P\left(\frac{\ell_i}{\bar\ell_i}>1 + \frac{14}{\sqrt{1-\alpha}}\sqrt{\frac{\log(n/\delta)}{Td_{\min}}}; L_i^c\right) \le \frac{\delta}{6n}, \\ \nonumber
    &(iii)\,\, \mathbb P\left(\frac{\ell_i}{\bar\ell_i}<1 -\frac{14}{\sqrt{1-\alpha}}\sqrt{\frac{\log(n)}{Td_{\min}}}; L_i^c\right) \le \frac{\delta}{6n}.
\end{align}
  The proof of (e) follows easily if we use the union bound and the probability bounds of \eqref{ingrad}(i)-(iii). Now we prove these bounds.

To prove \eqref{ingrad}(i), we bound $\mathbf P(K_{ij}^\pm)$ and $\mathbf P(J_{i}^\pm)$ separately. Now, if $T\le \frac{\log(n)}{1-\alpha}$, then $\mathbb P(K_{ij}^\pm)=0$, as $T_{ij}\in[0, T]$ and $T\le \frac{1}{1-\alpha}\sqrt{T(1-\alpha)\log(n)}$.
 On the other hand,  if $T> \frac{\log(n)}{1-\alpha}$, then taking $x=\frac{4-3\alpha}{1-\alpha}\sqrt{\log(n/\delta)/T(1-\alpha)}$ and using the standard binomial upper-tail large deviation estimate, we have that
\begin{align*}
 \mathbb P(K_{ij}^\pm) \le \mathbb P\left(|\bkt{T_{ij}}|>x\E T_{ij}\right) \le \exp\left(-\frac{T(1-\alpha)x^2}{2+x}\right) &\le \exp\left(-\frac{\left(3+\frac{1}{1-\alpha}\right)^2\log(n/\delta)}{2+\left(3+\frac{1}{1-\alpha}\right)}\right) \\
 &\le \delta^2/n^2.
 \end{align*}
So, using the union bound, 
\begin{equation} \label{Kij bd}
    \mathbb P\left(\cup_{\star\in\{+,-\}, j\in[n]}K_{ij}^\star\right) \le \frac{\delta}{12n}
\end{equation}
if $n$ is large enough.

To bound $\mathbb P(J_i^\pm)$, we take 
 $x=\sqrt{5\log(n/\delta)/(\bar\ell_i(1-\alpha))}$, and consider the functions $\gamma_{j,\pm}(u):=(1\pm u)\log(1\pm u)p_{ij}-(1\pm u)^{ p_{ij}}+1$ for $u\ge 0$. Clearly, $\gamma_{j,\pm}(0)=\gamma_{j,\pm}'(0)=0$, $\gamma_{j,+}''(u^*)\ge\gamma_{j,+}''(u)\ge p_{ij}/(1+u) \ge p_{ij}/2$ for all $u^*\in(0, u)$ where $u$ is small enough, and $\gamma_{j,-}''(u^*)\ge\gamma_{j,-}''(0)\ge p_{ij}$ for all $u^*\in(0, u)$. So, using a standard analysis argument, we get $\gamma_{j,\pm}(x) \ge  p_{ij} x^2/4$ if $Td_{\min}>C(\alpha)\log(n)$ and $C$ is large enough. Noting that $\E e^{\theta T_{ij}} = [\alpha+(1-\alpha) e^\theta]^T\le \exp(T(1-\alpha)(e^\theta-1))$ for all $\theta\in\mathbb R$ and $j\in[n]$, and using Markov inequality with $\theta_\pm=\log(1\pm x)$, we have that
\begin{align} \nonumber
\mathbb P(J_i^+) &\le \mathbb P\bigg(\sum_{j\in[n]}T_{ij}p_{ij}>(1+x)\bar\ell_i(1-\alpha)\bigg) \\ \nonumber
&\le \E\exp\bigg(\theta_+\sum_{j\in[n}T_{ij}p_{ij}-\theta_+(1+x)\bar\ell_i(1-\alpha)\bigg) \\
&\le \exp\bigg(-T(1-\alpha)\sum_{j\in[n]}\gamma_{j,+}(x)\bigg) \le e^{-\bar\ell_i(1-\alpha) x^2/4} \le (\delta/n)^{5/4}\le \delta/(24n) \label{Ji bd1}
\end{align}
if $n$ is large enough. Also,
\begin{align} \nonumber
 \mathbb P(J_i^-) &\le \mathbb P\bigg(\sum_{j\in[n]}T_{ij}p_{ij}<(1-x)\bar\ell_i(1-\alpha)\bigg) \\ \nonumber
 &\le \E\exp\bigg(-\theta_-\sum_{j\in[n}T_{ij}p_{ij}+\theta_-(1-x)\bar\ell_i(1-\alpha)\bigg) \\
&\le \exp\bigg(-T(1-\alpha)\sum_{j\in[n]}\gamma_{j,-}(x)\bigg) \le e^{-\bar\ell_i(1-\alpha) x^2/4} \le (\delta/n)^{5/4} \le \delta/(24n) \label{Ji bd2} 
\end{align}
if $n$ is large enough. Combining \eqref{Kij bd}, \eqref{Ji bd1},  and \eqref{Ji bd2} with the union bound proves \eqref{ingrad}(i). 
 
 To prove \eqref{ingrad}(ii) (resp.~(iii)), let $\lambda=\sqrt{(1-\alpha)\log(n/\delta)/Td_{\min}}$ and $\cR$ be the region
\begin{align*}
    \cR := \Bigg\{\mathbf U &= (U_1, \ldots, U_n)\in\mathbb R^n : \\
    &|U_j-T(1-\alpha)|\le \frac{4-3\alpha}{1-\alpha}\sqrt{T(1-\alpha)\log(n/\delta)} \text{ for all } j \in [n], \text { and } \\ &\qquad\qquad\bigg|\sum_{j\in[n]}U_j p_{ij} -\bar\ell_i(1-\alpha)\bigg| \le \sqrt{3\bar\ell_i(1-\alpha)\log(n/\delta)}\Bigg\}
\end{align*}
and for $\mathbf{U} \in \cR$,
\[
    h(\mathbf{U}) := \sum_{j\in[n], k\in[U_j]} G_{ij}^{(k)}I_{ij}^{\big(S_{ij}^{(k-1)}\big)},
\]
where $(I_{ij}^{(k)}, k \ge 1)$ and $(G_{ij}^{(k)}, k \ge 1)$ are independent sequences of i.i.d.~$\mathrm{Bernoulli}(p_{ij})$ and $\mathrm{Geometric}(1-\alpha)$ random variables, respectively, and $S_{ij}^{(k)}=\sum_{m\in[k]}G_{ij}^{(m)}$. Since $\cR$ is a bounded set, we can find $\check{\mathbf U}\in\cR$ (resp. $\hat{\mathbf U}\in\cR$) that maximizes $\E e^{\lambda h(\mathbf U)}\mathbf 1_\cR$ (resp.~$\E e^{-\lambda h(\mathbf U)}\mathbf 1_\cR$). Combining this observation with the facts (a) $h(T_{i1}, \ldots, T_{in}) \pre \ell_i \pre h(T_{i1}+1, \ldots, T_{in}+1)$, (b) $\sum_{j\in[n]}\check U_jp_{ij} \le \bar\ell_i(1-\alpha)+\sqrt{3\bar\ell_i(1-\alpha)\log(n/\delta)}$, and (c) $\sum_{j\in[n]}\hat U_jp_{ij} \ge \bar\ell_i(1-\alpha)- \sqrt{3\bar\ell_i(1-\alpha)\log(n/\delta)}$, and recalling the functions $\psi(\cdot)$ and $\phi(\cdot)$  we see that
\begin{align}
  \E e^{\lambda \ell_i}\mathbf 1_{L_i^c}&\le \E e^{\lambda h(\check{\mathbf U})}\mathbf 1_\cR \le \E e^{\lambda h(\check{\mathbf U})} \le \E\exp\left(\lambda\sum_{j\in[n], k\in[\check{\mathbf U}]} I^{(k)}_{ij}G^{(k)}_{ij}\right)\notag\\
   &\le\prod_{j\in[n]}\left[(1-p_{ij})+p_{ij}\frac{1-\alpha}{e^{-\lambda}-\alpha}\right]^{\check U_j}\le \exp\left(\sum_{j\in[n]}\check U_jp_{ij}\frac{\psi(\lambda)}{1-\alpha}\right)\notag\\
   &\le \exp\left(\bar\ell_i\psi(\lambda)\left[1+\sqrt{\frac{3\log(n/\delta)}{\bar\ell_i(1-\alpha)}}\right]\right),\text{ and}\label{Ucheck bd}\\
   \E e^{-\lambda \ell_i}\mathbf 1_{L_i^c}&\le \E e^{-\lambda h(\hat{\mathbf U})}\mathbf 1_\cR \le \E e^{-\lambda h(\hat{\mathbf U})} \le \E\exp\left(\sum_{j\in[n], k\in[\hat{\mathbf U}]} I^{(k)}_{ij}G^{(k)}_{ij}\right) \notag \\
  &\le\prod_{j\in[n]}\left[(1-p_{ij})+p_{ij}\frac{1-\alpha}{e^{\lambda}-\alpha}\right]^{\hat U_j} \le \exp\left(-\sum_{j\in[n]}\hat U_jp_{ij}\frac{\phi(\lambda)}{1-\alpha}\right)\notag\\
  &\le \exp\left(-\bar\ell_i\phi(\lambda)\left[1-\sqrt{\frac{3\log(n/\delta)}{\bar\ell_i(1-\alpha)}}\right]\right)\label{Uhat bd}
\end{align} 
Next we note that $\lambda<(1-\alpha)/2$ if $\frac{Td_{\min}}{\log(n)}$ is large enough; so (see \eqref{psi'bd} for details)
\[
    \psi(\lambda)\le\lambda\psi'(\lambda) \le \lambda(1+6\lambda/(1-\alpha)),
\]
and hence
\begin{align*}
   \psi(\lambda)\left[1+\sqrt{\frac{3\log(n/\delta)}{\bar\ell_i(1-\alpha)}}\right] &\le \lambda\left(1+\frac{6\lambda}{1-\alpha}\right) 
   \left[1+\sqrt{\frac{3\log(n/\delta)}{\bar\ell_i(1-\alpha)}}\right] \\
  &\le \lambda\left( 1+\frac{12\lambda}{1-\alpha}+\sqrt{\frac{3\log(n/\delta)}{\bar\ell_i(1-\alpha)}}\right) \\
   &\le \lambda \left(1+\frac{12+\sqrt 3}{\sqrt{1-\alpha}}\sqrt{\frac{\log(n/\delta)}{Td_{\min}}}\right). 
\end{align*}
Using the above estimate, \eqref{Ucheck bd}, and Markov inequality, if $\kappa(\alpha)=\frac{16}{\sqrt{1-\alpha}}$, then
\begin{align*}
 & \mathbb P\left(\ell_i>\bar\ell_i+
 \kappa\bar\ell_i\sqrt{\log(n/\delta)/Td_{\min}}; L_i^c\right) 
  \le \frac{\E e^{\lambda \ell_i}\mathbf 1_{L_i^c}}{\exp(\lambda\bar\ell_i+\kappa\lambda\bar\ell_i\sqrt{\log(n/\delta)/Td_{\min}})} \\
&  \le \exp\left(\bar\ell_i\psi(\lambda)\left[1+\sqrt{\frac{3\log(n/\delta)}{\bar\ell_i(1-\alpha)}}\right]-\lambda\bar\ell_i-\kappa\bar\ell_i\lambda\sqrt{\log(n/\delta)/Td_{\min}}\right) \\
&  \le \exp\left(-\frac{2}{\sqrt{1-\alpha}}\lambda\bar\ell_i \sqrt{\frac{\log(n/\delta)}{Td_{\min}}}\right) \le (\delta/n)^2 \le \delta/(6n)
\end{align*}
if $n$ is large enough. This proves \eqref{ingrad}(ii). Next from \eqref{phi'bd}, we get 
\[
     \phi(\lambda)\ge\lambda\phi'(\lambda) \ge \lambda(1-2\lambda/(1-\alpha)),
\]
and hence
\begin{align*}
 \phi(\lambda)\bigg[1-\sqrt{\frac{3\log(n/\delta)}{\bar\ell_i(1-\alpha)}}\bigg] 
   &\ge \lambda\bigg(1-\frac{\lambda}{1-\alpha}\bigg)\bigg[1-\sqrt{\frac{3\log(n/\delta)}{\bar\ell_i(1-\alpha)}}\bigg]\\
   &\ge \lambda\bigg[1-\frac{\lambda}{1-\alpha}-\sqrt{\frac{3\log(n/\delta)}{\bar\ell_i(1-\alpha)}}\bigg] \ge \lambda\bigg[1-\frac{1+\sqrt 3}{\sqrt{1-\alpha}}\sqrt{\frac{\log(n/\delta)}{Td_{\min}}}\bigg].
\end{align*} 
Using the above estimate, \eqref{Uhat bd} and Markov inequality, 
\begin{align*}
 & \mathbb P\left(\ell_i<\bar\ell_i-\kappa\bar\ell_i\sqrt{\log(n/\delta)/Td_{\min}}; L_i^c\right) 
  \le \frac{\E e^{-\lambda \ell_i}\mathbf 1_{L_i^c}}{\exp(-\lambda\bar\ell_i+\kappa\lambda\bar\ell_i\sqrt{\log(n/\delta)/Td_{\min}})} \\
 & \le  \exp\left(-\bar\ell_i\phi(\lambda)\left[1-\sqrt{\frac{3\log(n/\delta)}{\bar\ell_i(1-\alpha)}}\right]+\lambda\bar\ell_i-\kappa\bar\ell_i\lambda\sqrt{\log(n/\delta)/Td_{\min}}\right) \\
&  \le \exp\left(-\frac{2}{\sqrt{1-\alpha}}\lambda\bar\ell_i \sqrt{\frac{\log(n/\delta)}{Td_{\min}}}\right) \le (\delta/n)^2 \le \delta/(6n)
\end{align*}
if $n$ is large enough. This proves \eqref{ingrad}(iii) and the proof of (e) is complete.
\end{proof}
\begin{proof}[Proof of Theorem~\ref{thm:comm}]
The proof has two parts. First, one uses the Davis-Kahan theorem to derive an upper bound on the misclustering error in terms of $\|A - \bar{A}\|$ (resp. $\|\mathcal{L} - \bar{\mathcal{L}}\|$) when the $(1 + \epsilon)$-approximate spectral clustering algorithm is applied on the aggregated adjacency matrix (resp. the aggregated Laplacian matrix). For example, using Lemmas 5.1 and 5.3 of \cite{lei2015consistency} (see also \cite{bhattacharyya2020consistentrecovery}), we see that
\begin{align*}
    \err(\hat{Z}_{\adj}, Z) \le \tilde{C} (2 + \epsilon)\frac{K \|A - \bar{A}\|^2}{\gamma_{\min}(\bar{A})^2}   
\end{align*}
for some absolute constant $\tilde{C} > 0$, where $\gamma_{\min}(\bar{A})$ is the smallest non-zero singular value of $\bar{A}$. A similar bound holds in case of $(1 + \epsilon)$-approximate spectral clustering on the aggregated Laplacian matrix:
\begin{align*}
    \err(\hat{Z}_{\lap}, Z) \le \tilde{C} (2 + \epsilon)\frac{K \|\mathcal{L} - \bar{\mathcal{L}}\|^2}{\gamma_{\min}(I_n - \bar{\mathcal{L}})^2},   
\end{align*}
where $\gamma_{\min}(I_n - \bar{\mathcal{L}})$ is the smallest non-zero singular value of $I_n - \bar{\mathcal{L}}$.

In the second step, we upper bound $\|A - \bar{A}\|$ or $\|\mathcal{L} - \bar{\mathcal{L}}\|$ using Theorem~\ref{thm:main}. Then the desired results follow from combining the two steps and noting that
\[
    \gamma_{\min}(\bar{A}) = T \gamma_{\adj}, \text{ and } \gamma_{\min}(I_n - \bar{\mathcal{L}}) = \gamma_{\lap},
\]
which follow from the fact that $\bar{A} = TP$.
\end{proof}

\begin{proof}[Proof of Theorem~\ref{thm:cpe}]
Let 
\[
    \bar{S}_{\xi}(t) := \bigg(\frac{t}{T}\bigg(1 - \frac{t}{T}\bigg)\bigg)^{\xi} \bigg\|\frac{1}{t}\sum_{s = 1}^t \E A^{(s)} - \frac{1}{T - t}\sum_{s = t + 1}^T \E A_s \bigg\|,
\]
It is easy to see that
\[
    \bar{S}_{\xi}(t) = \begin{cases}
        \frac{t^{\xi}}{(T - t)^{1 - \xi}} \frac{T - \tau}{T^{2\xi}} \|P - Q\| & \text{if } t \le \tau, \\
        \frac{(T - t)^{\xi}}{t^{1 - \xi}} \frac{\tau}{T^{2\xi}} \|P - Q\| & \text{if } t > \tau.
    \end{cases}
\]
It is easy to check that the function $t \mapsto \bar{S}_{\xi}(t)$ is strictly increasing up to $\tau$, beyond which it is strictly decreasing. Now the same argument as in the proof of Lemma 5.6.1 in \cite{mukherjee2018thesis}, yields that 
\[
    \P(|\hat{\tau}_{\cusum}^{(\Lambda, \xi)} - \tau| > \eta T) \le 2\P(\max_{t \in [\Lambda, T - \Lambda]}|S_{\xi}(t) - \bar{S}_{\xi}(t)| > \frac{\eta}{4} \bar{S}_{\xi}(\tau)).
\]
We note that
\[
    |S_{\xi}(t) - \bar{S}_{\xi}(t)| \le \bigg(\frac{t}{T}\bigg(1 - \frac{t}{T}\bigg)\bigg)^{\xi} \bigg\|\frac{1}{t}\sum_{s = 1}^t (A^{(s)} - \E A^{(s)}) - \frac{1}{T - t}\sum_{s = t + 1}^T (A^{(s)} - \E A^{(s)} \bigg\|.
\]
Following the proof of Lemma 5.6.2 in \cite{mukherjee2018thesis}, we see that
\[
    |S_{\xi}(t) - \bar{S}_{\xi}(t)| \le \bigg(\frac{t}{T}\bigg(1 - \frac{t}{T}\bigg)\bigg)^{\xi} \frac{2T}{\Lambda} \epsilon,
\]
with probability at least $1 - 3 \delta_{\Lambda}(\epsilon)$, provided we have a concentration inequality for lazy IER processes of the form
\[
    \P(\frac{1}{t}\|\sum_{i = 1}^t (A^{(s)} - \E A^{(s)})\| > \epsilon) \le \delta_{\Lambda}(\epsilon),
\]
for $t \ge \Lambda$. By Theorem~\ref{thm:main}, we have that
\[
    \P(\|\sum_{i = 1}^T (A^{(s)} - \E A^{(s)}\|) > \epsilon) \le n \exp\bigg(-\frac{\epsilon^2}{C_1^2 T d_{\max}}\bigg) = \delta(\epsilon),
\]
provided $Td_{\max} > C(\log n)^3$ and $\delta(\epsilon) \in (n^{-c}, \frac{1}{2})$. We assume that $\Lambda \min(d^P_{\max}, d^Q_{\max}) > C (\log n)^3$. Then for our lazy IER process with changepoint, we have a concentration inequality
\[
    \P(\frac{1}{t}\|\sum_{i = 1}^t (A^{(s)} - \E A^{(s)})\| > \epsilon) \le \delta_{\Lambda}(\epsilon) = n \exp\bigg(-\frac{\Lambda\epsilon^2}{C_1^2 \max(d^P_{\max}, d^Q_{\max})}\bigg) =: \delta_{\Lambda}(\epsilon),
\]
for all $t \ge \Lambda$, provided that $\delta_{\Lambda}(\epsilon) \in (n^{-c}, 1/2)$. Combining everything together, we conclude that
\[
    |\hat{\tau}_{\cusum}^{(\Lambda, \xi)} - \tau| \le \eta T
\]
with probability at least
\[
    1 - 6T\delta_{\Lambda}(\frac{\eta}{4}\bar{S}_{\xi}(\tau)) = 1 - 6Tn \exp\bigg(-\frac{\Lambda\eta^2\big(\frac{\tau(T - \tau)}{T^2}\big)^{2\xi}\|P - Q\|^2}{16C_1^2\max(d^P_{\max}, d^Q_{\max})}\bigg),
\]
where we have used the fact that $\bar{S}_{\xi}(\tau) = \big(\frac{\tau(T - \tau)}{T}\big)^{\xi}\|P - Q\|$. We conclude that for any $c' < c$ we have
\[
    |\hat{\tau}_{\cusum}^{(\Lambda, \xi)} - \tau| \le \frac{4 C_1 T}{\big(\frac{\tau(T - \tau)}{T}\big)^{\xi}\|P - Q\|} \sqrt{\frac{\max(d^P_{\max}, d^Q_{\max})}{\Lambda} (1 + c')\log n}
\]
with probability at least $1 - 6Tn^{-c'}$.
\end{proof}


\bibliographystyle{apalike}
\bibliography{bib.bib}

\end{document}